\documentclass{amsart}

\usepackage{geometry,graphicx,amssymb,amsmath,amsthm,amsbsy,eucal,amsfonts,mathrsfs,amscd,bm,booktabs,tcolorbox, xcolor,caption,float,tikz,tikz-3dplot,cancel,cases,stmaryrd,wasysym, scalerel}
\usepackage{siunitx}
\usepackage[font=small]{subcaption}
\usetikzlibrary{patterns}
\usetikzlibrary{arrows.meta}
\usetikzlibrary{calc}
\usetikzlibrary{cd}
\usepackage[textsize=tiny,color=blue!50!white]{todonotes}
\usepackage[normalem]{ulem} % for strikethrough (\sout) in collaborative writing

\usepackage[breaklinks,bookmarks=false]{hyperref}

\hypersetup{colorlinks, linkcolor=blue, citecolor=blue,
urlcolor=blue, plainpages=false, pdfwindowui=false,
pdfstartview={FitH}, pdftitle={},}

\numberwithin{equation}{section}

\allowdisplaybreaks[3]

\newtheorem{theorem}{Theorem}[section]
\newtheorem{lemma}[theorem]{Lemma}
\newtheorem{assumption}[theorem]{Assumption}
\newtheorem{corollary}[theorem]{Corollary}
\newtheorem{proposition}[theorem]{Proposition}
\newtheorem{definition}[theorem]{Definition}
\newtheorem{remark}[theorem]{Remark}

\usepackage{mydef}

\begin{document}

\title[A Framework for Analysis of DEC Approximations]{A Framework for Analysis of DEC Approximations to Hodge-Laplacian Problems using Generalized Whitney Forms}

\def\JG{Johnny Guzm\'an}
\def\JGa{Division of Applied Mathematics,
Brown University,
Box F,
182 George Street,
Providence, RI 02912, USA}

\def\PP{Pratyush Potu}
\def\PPa{Division of Applied Mathematics,
Brown University,
Box F,
182 George Street,
Providence, RI 02912, USA}

\author{\JG, \PP}
% \author{\JG}
% \address{\JGa}
% \email{\href{mailto:johnny_guzman@brown.edu}{johnny\_guzman@brown.edu}}

% \author{\PP}
% \address{\PPa}
% \email{\href{mailto:pratyush_potu@brown.edu}{pratyush\_potu@brown.edu}}

%\thanks{}

% \subjclass[2020]{65N30}
% \keywords{hodge laplacian, discrete exterior calculus}

\begin{abstract}
We provide a framework for interpreting Discrete Exterior Calculus (DEC) numerical schemes in terms of Finite Element Exterior Calculus (FEEC). We demonstrate the equivalence of cochains on primal and dual meshes with Whitney and generalized Whitney forms which allows us to analyze DEC approximations using tools from FEEC. We demonstrate the applicability of our framework by rigorously proving convergence with rates for the Hodge-Laplacian problem in full $k$-form generality on well-centered meshes. We also provide numerical results illustrating optimality of our derived convergence rates. Moreover, we demonstrate how superconvergence phenomena can be explained in our framework with corresponding numerical results.
\end{abstract}

\maketitle

%%%%%%%%%%%%%%%%%%%%%%%%%%%%%%%%%%%%%
%%%%%%%%%%%%%%%%%%%%%%%%%%%%%%%%%%%%%

\section{Introduction}\label{sec:intro}
The Discrete Exterior Calculus (DEC) of \cite{Desbrun05, Hirani03} provides discrete analogues to various tools from the calculus of differential forms on smooth manifolds. The fundamental objects in DEC consist of discrete versions of differential forms, exterior derivatives, and Hodge stars which are defined over simplicial partitions (meshes) of smooth manifolds. A key aspect of DEC is the use of a dual mesh to define these fundamental objects. One can then construct a variety of differential geometric operations and operators including but not limited to discrete wedge products, Lie derivatives, and Laplace-Beltrami and Hodge-Laplacian operators \cite{Hirani03}. In this sense, DEC is one method of constructing a theory of discrete differential geometry (see \cite{Crane17} for a more comprehensive overview of discrete differential geometry). This has lead naturally to applications in digital geometry and computer graphics \cite{Crane13, Desbrun03}.

On the other hand, the expression of partial differential equations (PDEs) in the language of exterior calculus and differential forms is well-known. Hence, the objects in DEC have also found use in approximation schemes for PDEs. For instance, applications in fluid mechanics \cite{Griebel17, Hirani03, Mohamed16, Nitschke17, Pavlov11, Wang23}, elasticity \cite{Yavari08}, electromagnetism \cite{Stern15}, and more recently multiphysics simulations \cite{Morris24} have been realized. 

Notably, the convergence of DEC approximations for PDEs have been observed in numerical experiments (e.g. in \cite{Hirani15, Mohamed16, Mohamed18, Nitschke17, Schulz20}). However, rigorous analysis and proof of this convergence behavior is comparatively lacking. To this end, one strategy is the comparison or reinterpretation of DEC based schemes with other structure preserving numerical methods. For example, in \cite{Hildebrandt06}, a DEC scheme for the Poisson problem in two dimensions is equated with a Finite Element Method and in \cite{Griebel17}, modified DEC schemes are equated with Finite Volume and Finite Difference Schemes. Conversely, it was shown in \cite{Eldred22} that TRiSK-type schemes for shallow water equations can be reinterpreted as DEC based schemes.

A natural PDE to consider in smooth exterior calculus is that of the Hodge-Laplacian. In fact, many of the PDEs for which DEC has been applied to are special cases of the Hodge-Laplacian problem. Thus, a convergence analysis of DEC based approximations for the Hodge-Laplacian problem is particularly desirable. In this direction, the most significant result is \cite{Schulz20} where the analysis is based on the development of an appropriate discrete framework for understanding convergence in arbitrary dimension. However, the convergence results are only realized for the 0-form Hodge-Laplacian (Poisson problem). 
We note that at the same time as this work was in review, convergence of the DEC scheme for the Hodge-Laplacian problem in 2 dimensions for all $k$-forms was independently proven in \cite{Zhu2025} by showing the solution from the DEC scheme converges to the solution of the FEEC scheme which in turn is known to converge to the solution of the continuous problem (see e.g. \cite{Arnold2006}). This strategy is noticeably different from the analysis in this work.

The primary contribution of this paper is to provide a novel framework for analysis of DEC approximations to Hodge-Laplacian problems which allows for rigorous proofs of convergence with optimal rates. It should be noted that we only consider \emph{well-centered} meshes in this work. While constructing such meshes may be difficult in practice, algorithms to produce well-centered meshes have been studied \cite{Vanderzee09}. A key ingredient in this framework is the existence of complexes of differential forms on cellular meshes developed in \cite{Christiansen08, Christiansen09} and \cite[Section 5]{Christiansen11} which allows us to analyze Hodge-Laplacian problems for general differential $k$-forms in arbitrary dimension $n$. Moreover, the framework also provides conditions for superconvergence phenonemona on certain symmetric meshes which has been observed numerically. We provide numerical experiments which realize both the optimal convergence rates and superconvergence.

We remark that the generalized Whitney forms are not necessary to analyze the DEC scheme. One can perform the entirety of the analysis working over cochains by replacing appropriate constructions with their cochain counterparts. Nevertheless, we opt for a differential form-based presentation to showcase connections with the well-established FEEC framework and use corresponding methodologies. Our hope is that this approach naturally reveals deeper connections and opens avenues for future analysis.

The paper is organized as follows. We briefly introduce the continuous level setting of differential forms as well as the Hodge-Laplacian problem in Section \ref{sec:continuous_setting}. We then introduce the discrete setting along with the standard construction of DEC and the DEC scheme in Section \ref{sec:cochain_setting}. In Section \ref{sec:WhitneyForms}, we reformulate DEC and in particular, the DEC scheme in terms of Whitney/Generalized Whitney forms. We then introduce useful tools and establish relevant results in Section \ref{preliminarytools} before finally stating and proving our main result on error estimates for the DEC approximation to the Hodge-Laplacian problem in Section \ref{sec:StabilityErrorEstimates}. We detail how our framework also allows for proofs of superconvergence under symmetry in Section \ref{sec:superconvergence}. We conclude with numerical experiments in Section \ref{sec:numerics}.

\section{Continuous setting}\label{sec:continuous_setting}

In this section, we provide a brief overview of differential forms, the Hodge star operation, exterior derivative, and codifferential. We also make note of the relevant infinite dimensional spaces of differential forms. We conclude by introducing the Hodge-Laplacian and the relevant problem.

We let $\Omega \subset \mathbb{R}^n$ be a bounded Lipschitz polytopal domain. We do not assume $\Omega$ is contractible so the spaces of harmonic forms may be nontrivial. We denote by $\La^k(\Omega)$, the space of differential $k$-forms with smooth coefficients on $\Omega$. The larger space allowing $L^2$ coefficients is denoted by $L^2 \La^k(\Omega)$ and similarly $H^{\ell} \La^k(\Omega)$ denotes the space of $k$-forms with coefficients in the Sobolev space $H^{\ell}(\Omega)$. The exterior derivative, denoted by $d^k$, maps $\La^k (\Omega) \rightarrow \La^{k+1}(\Omega)$ and extends more generally to the less regular spaces. We define the spaces
\begin{equation*}
H\La^k(\Omega)=\{\, v \in L^2 \La^k(\Omega)\,:\, d^k v \in L^2\La^{k-1}(\Omega)\,\},
\quad \mathring{H}\La^k(\Omega)=\{\, v \in H\La^k(\Omega)\,:\, \tr_{\partial \Omega} v=0\,\},
\end{equation*}
where $\tr$ is the trace operator. We may omit the denotation of $\Omega$ when this is clear from context.

Now, we let $\{e_1,\dots, e_n\}$ be the canonical basis of $\R^n$, and refer to the dual basis of $\{e_1,\dots, e_n\}$ as $\{dx_1,\dots,dx_n\}$. We denote the volume form by $\vol = dx_1\wedge dx_2 \wedge \dots \wedge dx_n$.

Let $\Sigma(k,n)$ be the set of strictly increasing maps from $\{1,\dots, k\}$ to $\{1,\dots, n\}$. That is, an element of $\Sigma(k,n)$ is of the form $I=(i_1, i_2, \dots, i_k)$ where $i_1,i_2,\dots,i_k$ are integers such that $1\leq i_1<i_2<\cdots<i_k\leq n$. We use the notation $dx_I = dx_{i_1}\wedge dx_{i_2}\wedge \cdots \wedge dx_{i_k}$ when $I = (i_1,i_2,\dots,i_k)$.

Now, if $\omega, \rho\in \Lambda^k(\Omega)$, we can write $\omega(z) = \sum_{I\in\Sigma(k,n)} a_I(z) dx_I$, $\rho(z)=\sum_{I\in \Sigma(k,n)}b_I(z)dx_I$, and the inner product between $k$-forms at a point $z\in\Omega$ is given by 
\begin{equation*}
    \llangle \omega(z), \rho(z)\rrangle = \sum_{I\in\Sigma(k,n)} a_I(z)b_I(z).
\end{equation*}
The Hodge star operator $\star$ maps $L^2\Lambda^k$ isomorphically onto $L^2\Lambda^{n-k}$ for each $k$. We fix an orientation on $\mathbb{R}^n$ so that we have the relation:
\begin{equation}\label{orientationrelation}
    (\omega \wedge \rho) (z) = \llangle \star\omega(z), \rho(z)\rrangle \vol \qquad \forall \rho\in \La^{n-k}(\Omega), z\in\Omega.
\end{equation}

We define the formal adjoint of $d^{k-1}$ by:
\begin{equation*}
 \de_k \omega= (-1)^k \star^{-1} d^{n-k}( \star \omega) \qquad \text{for } \omega \in \La^k(\Omega),
\end{equation*}
and the spaces
\begin{equation*}
H_{\delta}\La^k(\Omega)=\{\, v \in L^2 \La^k(\Omega)\,:\, \delta_k v \in L^2\La^{k-1}(\omega)\,\},
\quad \mathring{H}_{\delta}\La^k(\Omega)=\{\, v \in H_{\delta}\La^k(\Omega)\,:\, \text{tr}_{\partial \Omega} \star v=0\,\}.
\end{equation*}
The adjoint relation between $d$ and $\delta$  may be expressed as
\begin{equation}\label{intparts}
( d^{k-1} \omega, \rho)_{L^2 \La^{k}(\Omega)}= ( \omega, \de_{k} \rho)_{L^2 \La^{k-1}(\Omega)} \qquad \text{ for } \omega \in H\La^{k-1}(\Omega),\, \rho\in  \mathring{H}_{\delta} \La^{k}(\Omega),
\end{equation}
where $(\, \cdot\,,\, \cdot\,)_{L^2 \La^{k}(\Omega)}$ is the $L^2 \La^{k}(\Omega)$ inner product. We simply write $(\, \cdot\, ,\, \cdot \,)$ when the domain $\Omega$ is understood from the context.  We let $\| \omega\|_{L^2(\Omega)}^2= ( \omega,\omega)$ for $\omega\in L^2\La^{k}(\Omega)$. 

We also define the $H\La^k(\Omega)$ inner product as
\begin{equation*}
    (u,v)_{H_d}:=(u,v)+(du,dv),
\end{equation*}
with corresponding norm $\|\cdot\|_{H_d}$.

Moreover, for $\omega = \sum_{I\in\Sigma(k,n)}a_Idx_I\in H^\ell\La^{k}(\Omega)$, we let 
\begin{equation*}
    |\omega|_{H^j(\Omega)}^2:= \sum_{I\in\Sigma(k,n)} |a_I|_{H^j(\Omega)}^2, \qquad \| \omega\|_{H^\ell(\Omega)}^2:= \sum_{j=0}^\ell |\omega|_{H^j(\Omega)}^2
\end{equation*}
where $|\cdot|_{H^j(\Omega)}$ is the standard $H^j$ seminorm. Note that the Hodge star operator is an isometry from $L^2\La^k$ to $L^2\La^{n-k}$. Moreover, with these norms as defined, one readily sees that the Hodge star operator is an isometry from $H^\ell\La^k$ to $H^\ell\La^{n-k}$ for any $\ell$.

We define the space of harmonic forms as
\begin{equation*}
    \Hh^k=\{v\in H\La^k(\Omega)\cap \mathring{H}_{\delta}\La^k(\Omega): d^k v = 0,\, \delta_k v = 0\},
\end{equation*}
and let $P_{\Hh^k} u$ denote the $L^2\La^k$-orthogonal projection of $u$ onto $\Hh^k$ which satisfies
\begin{equation*}
    (P_{\Hh^k} u, v) = (u,v) \quad \forall v\in \Hh^k.
\end{equation*}

Following \cite{Arnold2006}, it will be useful to define the space
\begin{equation*}
    (\Zz H\La^k)^{\perp}:=\{u\in H\La^k(\Omega): (u,v) =0 \:\:\forall v\in H\La^k(\Omega) \text{ such that }dv = 0\},
\end{equation*}
and note that any $u\in L^2\La^k(\Omega)$ has the Hodge decomposition \cite[Section 2, Section 7.2]{Arnold2006}:
\begin{equation}\label{hodgedecomposition}
    u = u_d + u_\delta + u_{\Hh}.
\end{equation}
where $u_d\in d(H\La^{k-1}(\Omega))$, $u_\delta\in \delta (\mathring{H}_\delta\La^{k+1}(\Omega))$, and $u_{\Hh}\in\Hh^k$.

Now, the Hodge-Laplacian $\Delta^k$ is defined as
\begin{equation*}
    \Delta^k = \delta_{k+1} d^k + d^{k-1}\delta_k.
\end{equation*}

The problem of interest for us is the following.
\begin{subequations}\label{continuousstrongform}
    \begin{alignat}{2}
  \Delta^k u&=  f-P_{\Hh^k}f \quad &&\text{ in } \Omega,  \label{csf}\\
  P_{\Hh^k} u &= 0 \quad &&\text{ in } \Omega, \label{csfc}\\
  \tr \star u&=0 \quad &&\text{ on } \partial \Omega, \label{csfbc1}\\
  \tr  \star d u&=0 \quad &&\text{ on } \partial \Omega \label{csfbc2}.
    \end{alignat}
\end{subequations}
We will see that this is the natural problem to approximate using DEC. In particular, the problem \eqref{continuousstrongform} is indeed a well-posed problem \cite[Theorem 3.1, Section 3.2.1]{Arnold10}.

Finally, note that we can let $\rho = \delta_k u$ and $p=P_{\Hh}f$ to rewrite \eqref{csf} and \eqref{csfc} as
\begin{subequations}\label{continuousmixedform}
\begin{alignat}{1}
\rho-\delta_k u&= 0, \label{cmf1}\\
d^{k-1} \rho+ \delta_{k+1} d^k u+p&=f, \label{cmf2}\\
P_{\Hh^k} u &= 0.
\end{alignat}
\end{subequations}

\section{Discrete setting and DEC}\label{sec:cochain_setting}
In this section, we introduce the primal mesh which will be a simplicial complex. We also discuss polyhedral cell complexes and the dual mesh, and introduce the standard elements of DEC in terms of chains and cochains. A more detailed overview of DEC can be found in \cite{Desbrun05, Hirani03}.

\subsection{Primal mesh}
Let $\Th$ be an $n$-simplicial conforming mesh of $\Omega$ \cite[Definition 1.55]{ErnGuermond04}. We assume the shape regularity condition
\begin{equation*}
\frac{h_{\sig}}{\mu_{\sig}} \le C_S, \quad \forall \sig \in \Th,
\end{equation*}
where $\mu_{\sig}$ is the diameter of the largest inscribed ball in $\sig$, $h_{\sig}$ is the diameter of $\sigma$, and $C_S>0$ is the shape regularity constant. Associated with the mesh $\Th$ is the simplicial complex $\del(\Th)$ consisting of all the simplices of $\Th$ and all their subsimplices of dimension $0$ through $n$. We also use the notation $\del$ when the mesh $\Th$ is clear. We denote by $\del_k(\Th)$, or simply $\del_k$ when the mesh $\Th$ is clear, the collection of all the simplices in  $\del(\Th)$  of dimension $k$. We let $h:=\max_{\sig\in\Delta}h_\sig$. If $x_0,\ldots,x_k\in\mathbb{R}^n$ are the vertices of $\sig\in\del_k$, we may write $[x_0,\ldots,x_k]$ for $\sig$, the closed convex hull of the vertices.

To be precise a simplicial complex $\del$ is a set of simplices such that every subsimplex of a simplex of $\del$ is in $\del$ and the intersection of any two simplices in $\del$ is a subsimplex of both of them \cite[Section 2]{Munkres84}. Hence, one can see that $\del(\Th)$ is indeed a simplicial complex because $\Th$ is a conforming mesh.

We will need to consider orientations of simplices. This is a choice of ordering of the vertices with two orders differing by an even permutation giving the same orientation.  If we select an ordering of all the vertices of $\Th$, this implies a default orientation for each of the simplices in $\del(\Th)$. We let $x_1,\dots, x_N$ be such an enumeration of the vertices of $\Th$, and henceforth let $\del(\Th)$ (resp. $\del_k(\Th)$) consist of all simplices (resp. $k$-simplices) with their default orientations.

% Particular to DEC is the assumption on orientation of $n$-cells that two $n$-cells which share an $n-1$ face must have the same orientation. This is solely to obtain the relation \eqref{partialstarsig} noted in section \ref{sec:DEC_numerical_scheme}. We let $x_1, \dots, x_N$ be an enumeration of the vertices which satisfies this property.

Associated with the simplicial complex are a chain complex
and co-chain complex.  The space of $k$-chains is the vector space
\begin{equation*}
\C_k:= \left\{\, \sum_{ \sig \in \del_k} a_{\sig} \sigma\,:\,  a_\sig \in \mathbb{R} \,\right\},
\end{equation*}
where $\sig$ is given the default orientation and the same simplex with the opposite
orientation is identified with $-\sig$. 

It will be useful to introduce the following sets of oriented simplices for any $\sig\in\del_k$. For $ \sigma= [x_{i_0}, x_{i_1}, \ldots, x_{i_k}] \in \Delta_k$, we let 
\begin{alignat*}{1}
\mathcal{A}(\sigma)=&\{ \tau \in \Delta_{k+1}: \tau =[x, x_{i_0}, \ldots, x_{i_k}], x \in \Delta_0 \},\\    
\mathcal{B}(\sigma)=& \{ (-1)^j [ x_{i_0}, \ldots, \widehat{x_{i_{j}}}, \ldots, x_{i_k}]:  0 \le j \le k \},
\end{alignat*}
 where $\widehat{\cdot}$ indicates omission. In other words, $\mathcal{A}(\sigma)$ consists of all the $(k+1)$-simplices that have $\sigma$ as part of its boundary with a given orientation, and $\mathcal{B}(\sigma)$ is the collection of $(k-1)$-simplices which appear on the boundary of $\sigma$ with the appropriate orientation. 
 \begin{remark}\label{rem:A_B_relation}
      The two sets are related in that $\sigma \in \mathcal{B}(\tau)$ if and only if $\tau \in \mathcal{A}(\sigma)$. 
 \end{remark}
 
The boundary map $\partial_k:  \C_k \rightarrow \C_{k-1}$ is defined for $\sigma= [x_{i_0}, \ldots, x_{i_k}] \in \del_k$ by
\begin{equation}\label{boundarydef}
\partial_k \sigma= \sum_{\tau\in\mathcal{B}(\sigma)} \tau = \sum_{j=0}^k (-1)^j [x_{i_{0}}, \ldots,\widehat{x_{i_{j}}}, \ldots, x_{i_{k}}].
\end{equation}

The dual space of $\C_k$ is the space $\C^k$ of cochains. The coboundary operator $\dd^k: \C^k \rightarrow \C^{k+1}$ is defined by duality following the relation:
\begin{equation}\label{defdd}
\bl \dd^k \mathsf{w}, \tau \br =\bl \mathsf{w}, \partial_{k+1} \tau \br, \quad \tau \in \C_{k+1}, \mathsf{w} \in \C^k,
\end{equation}
where $\bl \cdot, \cdot\br$ is the duality pairing.

As a consequence of Remark \ref{rem:A_B_relation}, we have that for any $\mathsf{v}\in\C^k$ and $\mathsf{w}\in\C^{k+1}$, it holds
\begin{equation}\label{commutesum}
    \sum_{ \sigma \in \Delta_k} \sum_{\tau \in \mathcal{A}(\sigma)} \bl\mathsf{v}, \sigma\br \bl \mathsf{w}, \tau \br =  \sum_{\tau \in \Delta_{k+1}}  \sum_{\sigma \in \mathcal{B}(\tau)} \bl\mathsf{v}, \sigma\br \bl \mathsf{w}, \tau \br.
\end{equation}

\subsection{Dual mesh}\label{sec:dual_mesh}
A fundamental aspect of DEC is the so called dual mesh. Recall that the primal mesh $\Th$ has the associated simplicial complex $\Delta$. However, for the dual mesh, we will see that the associated complex will no longer be simplicial. Hence, we first discuss polyhedral cell complexes before turning to the specifics of the dual mesh.

\begin{definition}[Polyhedral cells]\label{def:polyhedralcell}
    A polyhedral $k$-cell is a subset of $\R^n$ which can be simplicially partitioned into a finite simplicial complex of dimension $k$ and is homeomorphic (there exists a continuous bijection with continuous inverse) to the closed unit ball of $\R^k$. 
\end{definition}

\begin{figure}[H]
\centering
\begin{subfigure}[b]{0.4\textwidth}
\centering
\begin{tikzpicture}[scale=1.4]

  \draw[->] (0,0) -- (2,0) node[below] {};
  \draw[->] (0,0) -- (0,2) node[left]  {};

  \coordinate (H) at (0.5,0.5);

  \coordinate (A) at (1.2,0.2);
  \coordinate (B) at (0.4,1.4);

  \draw[very thick] (H) -- (A) node[midway, above right]  {$\tau_1$};
  \draw[very thick] (H) -- (B) node[midway, above right]   {$\tau_2$};

\end{tikzpicture}
\caption{A single polyhedral 1-cell in $\R^2$}
\end{subfigure}
\hspace{0.5cm}
\begin{subfigure}[b]{0.4\textwidth}
\centering

\tdplotsetmaincoords{65}{111}

\begin{tikzpicture}[tdplot_main_coords, scale=1.4]

  %--- draw the 3D axes first
  \draw[->,thick] (0,0,0) -- (2,0,0) node[anchor=north east]{};
  \draw[->,thick] (0,0,0) -- (0,2,0) node[anchor=north west]{};
  \draw[->,thick] (0,0,0) -- (0,0,1.5) node[anchor=south]{};

  %--- now draw the two triangles (and labels) scaled up 1.5×
  \begin{scope}[scale=1.5]
    % original vertices
    \coordinate (P) at (0.1733, 0.7734, 0.5);
    \coordinate (Q) at (1.2887, 1.0597, 0.8);
    \coordinate (R) at (0.7616, 0.5252, 1.1);
    \coordinate (S) at (0.7624, 1.6789, 1.0);

    % fill & draw
    \filldraw[fill=blue!20,draw=black,thick]
      (P) -- (Q) -- (R) -- cycle;
    \filldraw[fill=blue!30,draw=black,thick]
      (P) -- (Q) -- (S) -- cycle;

    % two‐step barycentric to get centroids
    \coordinate (M)  at ($(P)!0.333!(Q)$);
    \coordinate (C1) at ($(M)!0.333!(R)$);
    \coordinate (C2) at ($(M)!0.333!(S)$);
    \node[font=\small\itshape] at (C1) {$\tau_1$};
    \node[font=\small\itshape] at (C2) {$\tau_2$};
  \end{scope}

\end{tikzpicture}
    \caption{A single polyhedral 2-cell in $\R^3$}
\end{subfigure}
\caption{Examples of our definition of polyhedral cells with explicit simplicial decomposition.}
\label{fig:polytopes}
\end{figure}

\begin{definition}[Polyhedral cell complex]
    A \textit{polyhedral cell complex}, $\delC$, of a domain $\Omega$ is a finite collection of polyhedral cells lying in $\Omega$ such that:
    \begin{enumerate}
        \item $\bigcup_{\sig\in\delC}\sig = \overline{\Omega}$.
        \item If $\sig,\sig'\in\delC$, then $\Int(\sig)\cap\Int(\sig') = \emptyset$.
        \item If $\sig,\sig'\in\delC$, then $\sig\cap \sig'= \bigcup_{\tau\in S} \tau$ for some $S\subset\delC$.
        \item If $\sig\in\delC$, then $\partial\sig= \bigcup_{\tau\in S} \tau$ for some $S\subset\delC$.
    \end{enumerate}
\end{definition}
The polyhedral cell complex defined above is a special case of a cellular complex \cite[Definition 5.1]{Christiansen11} where the cells are polyhedral. Thus, in what follows, we may simply refer to the elements of polyhedral cell complexes as cells. We let $\delC$ be a polyhedral cell complex, and the subset of $\delC$ containing all of its $k$-cells will be denoted $\delC_k$. We now introduce some notions paralleling those introduced in the preceding section.

We define an oriented polyhedral $k$-cell $\sigma$ to be a polyhedral $k$-cell such that its simplicial decomposition (see Definition \ref{def:polyhedralcell}) is consistently oriented: the affine mapping between any two constituent $k$-simplices has a positive Jacobian determinant.
The orientation of $\sigma$ is then given by the orientation of any simplex in its simplicial decomposition. Henceforth, we assume $\delC$ is an oriented polyhedral cell complex. When defining the dual mesh, we will make precise the orientation we choose for the cells of the relevant polyhedral cell complex.

It will be useful to define the subsets of $\delC$ corresponding to boundary and interior cells as $\delCB:=\{\sig\in\delC : \sig\subset\partial\Omega\}$ and $\delCI:=\delC\setminus \delCB$ respectively. Then the subsets of $\delCB$ and $\delCI$ consisting of all of the $k$-cells in $\delCB$ and $\delCI$ will be denoted $\delCB_{k}$ and $\delCI_{k}$ respectively.

We define the space of $k$-chains on the polyhedral cell complex as $\C_{\smp, k}$. As on the simplicial complex, this is the vector space
\begin{equation*}
\C_{\smp, k}:= \left\{\, \sum_{ \sig \in \delC_k} a_{\sig} \sig\,:\,  a_\sig \in \mathbb{R} \,\right\},
\end{equation*}
where $\sig$ has the particular orientation associated with $\delC$, and the same cell with the opposite orientation is identified with $-\sig$. Now, the boundary map, $\partial_k^{\smp}:  \C_{\smp, k} \rightarrow \C_{\smp, k-1}$, is defined for $\sig\in\delC_k$ by
\begin{equation*}
\partial_k^{\smp} \sig = \sum_{\substack{\tau\in \delC_{k-1} \\ \tau\subset\partial\sig}}\epsilon(\sig, \tau) \tau,
\end{equation*}
where $\epsilon$ is defined as follows. Let $\sig\in\delC_k$ and $\tau\in\delC_{k-1}$ such that $\tau\subset \partial \sig$ for $k\geq 1$. Then, if $k=1$, there exists a 1-simplex in the simplicial decomposition of $\sig$ which contains $\tau$, say $\widetilde{\sig}$. If $\widetilde{\sig}=[\tau, \tau']$ for some 0-simplex $\tau'$, we define $\epsilon(\sig, \tau)=-1$, and if $\widetilde{\sig}=[\tau', \tau]$ for some 0-simplex $\tau'$, we define $\epsilon(\sig, \tau)=+1$. For $k>1$, let $p$ be a point on $\tau$ with a well-defined $(k-1)$-tangent plane $T_p$. we define $\epsilon$ as in \cite[Appendix II, section 5]{Whitney57}. Recall that $\sig$'s orientation gives us an equivalence class of collections of $k$ basis vectors at $p$. We can thus choose a member of this class, $(v_1, \dots, v_k)$, such that $v_2, \dots, v_k$ lie in $T_p$. Then, if $v_1$ points out of $\sig$, we let $\epsilon(\sig,\tau)=+1$ and $\epsilon(\sig,\tau)=-1$ otherwise.

The dual space of $\C_{\smp, k}$ is $\C_{\smp}^k$, and the coboundary operator $\dd_{\smp}^k: \C_{\smp}^{k} \rightarrow \C_{\smp}^{k+1}$ is defined by duality as:
\begin{equation}\label{defdd2}
\bl \dd_{\smp}^k \mathsf{w}, \sig \br =\bl \mathsf{w}, \partial_{k+1}^{\smp} \sig \br, \quad \sig\in \delC_{k+1}, \mathsf{w} \in \C_{\smp}^{k}.
\end{equation}

We now turn to defining the dual mesh which is central to the DEC method. 
\begin{definition}[Orthogonal dual mesh]\label{defdualmesh}
    Let $\delC$ be a polyhedral cell complex of $\Omega$ such that for every simplex $\sig\in\del_{n-k}(\Th)$, there exists a unique dual cell $*\sig\in\delCI_{k}$ with the following properties.
\begin{enumerate}
    \item The mapping $(*)$ from $\del_{n-k}$ to $\delCI_k$ is an isomorphism.
    \item $*\sig$ lies completely in a $k$-plane, $P(*\sig)$.
    \item $P(\sig)\perp P(*\sig)$.
    \item Let $v_1, \dots, v_{n-k}$ be a collection of orthonormal vectors in the equivalence class corresponding to the orientation of $\sig$. Then, the orientation of $*\sig$ is such that 
    \begin{equation}\label{dual_orientation0}
        \det[v_1, \dots, v_{n-k}, w_1, \dots, w_k] = 1,
    \end{equation}
    where $w_{1}, \dots, w_{k}$ are orthonormal vectors in the equivalence class corresponding to the orientation of $*\sig$.
    Equivalently,
    \begin{equation}\label{dual_orientation}
        (dv_1\wedge\cdots \wedge dv_{n-k}) \wedge (dw_{1}\wedge \cdots \wedge dw_{k}) = \vol,
    \end{equation}
    where $dv_1, \dots, dv_{n-k}$ and $dw_1,\dots,dw_{k}$ are the dual bases of $v_1, \dots, v_{n-k}$ and $w_{1}, \dots, w_{k}$ respectively.
\end{enumerate}
Then, we say $\Dh:=\delC_n = \{*\sig\}_{\sig\in\del_0(\Th)}$ is an orthogonal dual mesh of $\Th$ and $\delC$ is the dual complex.
\end{definition}
\begin{definition}[Dual set]
    The dual set of $\del$ is exactly the subset of the dual complex $\delC$ consisting of dual cells of primal simplices. This is
\begin{equation*}
    \delC_{*} := \{ * \sigma: \sig \in \del\}.
\end{equation*}
The collection of all of the polyhedral $k$-cells in the dual set is denoted $\delC_{*,k}$.
\end{definition}

In general, it is difficult to satisfy all of the conditions in our assumptions on the orthogonal dual mesh. Hence, more general dual meshes exist in the literature. A natural one to consider is the barycentric dual mesh (e.g. see \cite{Bonelle15}). However, Properties 2 and 3 of Definition \ref{defdualmesh} may not be satisfied in general, and we will see that these properties are key assumptions in the proof of our error estimates in Section \ref{sec:StabilityErrorEstimates}. We can see from Property 1 of Definition \ref{defdualmesh} that $\delC_* = \delCI$. Note that the right hand side of \eqref{dual_orientation0} could either be $1$ or $-1$ due to the orthonormality of the noted vectors, and so we are effectively choosing the orientation of $*\sig$ in the natural way based on the orientation of $\sig$ and the orientation of $\R^n$ as specified in section \ref{sec:continuous_setting}. 

All of the four properties of Definition \ref{defdualmesh} are satisfied if one uses a mesh $\Th$ which is \textit{well-centered} (i.e. the circumcenter of any simplex in $\del$ lies in the interior of the simplex \cite[Definition 2.4.2]{Hirani03}) and one uses the circumcenters to define the dual mesh $\Dh$. This is detailed in \cite{Desbrun05, Hirani03}, and an algorithm to obtain $*\sig$ with the appropriate orientation is detailed in \cite[Remark 2.5.1]{Hirani03}.  

We note that even if the mesh $\Th$ belongs to a collection of well-centered meshes indexed by decreasing mesh size $h$, when discussing convergence, it is entirely possible that dual cells may degenerate (cf. \cite[Remark 2.4]{Zhu2025}) which leads to the DEC scheme becoming unstable. Hence, we need an assumption to disallow this possibility.
\begin{definition}[Uniformly well-centered collection of meshes]
    A collection of meshes $\{\Th\}_h$ is considered \emph{uniformly well-centered} if for any $k$-simplex $\sigma\in \Delta_k(\Th)$ where $k\in\{0, \dots, n\}$ for any $h$, the measure of the dual $(n-k)$-cell, denoted $|*\sig|$, is such that $|*\sig|\geq \beta > 0$ where $\beta$ is independent of $h$.
\end{definition}
\begin{assumption}\label{assump:UWC}
    In the rest of this paper, we assume $\Th$ belongs to a collection of  \emph{uniformly well-centered meshes} and $\Dh$ is the circumcentric dual mesh as detailed in \cite{Desbrun05, Hirani03} so we indeed have all of the properties in Definition \ref{defdualmesh}.
\end{assumption}

Briefly, the construction of the dual mesh in \cite{Desbrun05, Hirani03} is as follows. Let $\sig\in\del_k$. For notational convenience, let $\sig_k=\sig$. Then, there exists an ordered set of simplices $\{\sig, \sig_{k+1}, \dots, \sig_n\}$ such that $\sig_j\subset\sig_{j+1}$ for all $j\in\{k, \dots, n-1\}$ and $\sig_j\in\del_j$ for each $j\in\{k, \dots, n\}$. Let $\mathcal{K}(\sig)$ be the collection of all such ordered sets of simplices. Then, any $S\in \mathcal{K}(\sig)$ can be identified with an $(n-k)$-simplex $\tau_S:=[c(\sig), \dots, c(\sig_n)]$ where $c(\sig)$ denotes the circumcenter of $\sig$. Then, the cell $*\sig$ is exactly the union of all of the $\tau_S$ for $S\in \mathcal{K}(\sig)$ with the orientation noted in Definition \ref{defdualmesh}.

Following Definition \ref{defdualmesh}, the collection $\{*\sig\}_{\sig\in\del_0}$ is $\Dh$, and the associated polyhedral cell complex is $\delC$. We illustrate this oriented dual mesh in Figure \ref{fig:dualcellmesh}.

\begin{figure}[H]
\centering
\begin{subfigure}[b]{0.45\textwidth}
\centering
\begin{tikzpicture}[scale = 3, >={Latex[round]}] % Thicker lines

    \coordinate (A) at (0, 0);
    \coordinate (B) at (1, 0);
    \coordinate (C) at (1.5, 1);
    \coordinate (D) at (0.5, 1);

    \coordinate (C1) at (0.5, 0.375);
    \coordinate (C2) at (1, 0.625);

    % Outline main triangles with arrows at midpoints
    \draw (A) -- node[midway,sloped]{\tikz \draw[->] (0,0) -- (0.4,0);} (B);
    \draw (B) -- node[midway,sloped]{\tikz \draw[<-] (0,0) -- (0.4,0);} (D);
    \draw (D) -- node[midway,sloped]{\tikz \draw[<-] (0,0) -- (0.4,0);} (A);
    
    \draw (B) -- node[midway,sloped]{\tikz \draw[->] (0,0) -- (0.4,0);} (C);
    \draw (C) -- node[midway,sloped]{\tikz \draw[<-] (0,0) -- (0.4,0);} (D);
    
    % Draw the points
    \foreach \p in {A,B,C,D} {
        \fill[black] (\p) circle (0.75pt);
    }

    % Add \circlearrowleft at barycenters
    \node at (C1) {\scalebox{2}{$\circlearrowleft$}}; % Enlarged 2x
    \node at (C2) {\scalebox{2}{$\circlearrowleft$}}; % Enlarged 2x

\end{tikzpicture}
\caption{Oriented primal mesh $\Th$}
\end{subfigure}
\hfill
\begin{subfigure}[b]{0.45\textwidth}
\centering

\begin{tikzpicture}[scale = 3]

    \coordinate (A) at (0, 0);
    \coordinate (B) at (1, 0);
    \coordinate (C) at (1.5, 1);
    \coordinate (D) at (0.5, 1);
    
    \coordinate (M1) at (0.5, 0);
    \coordinate (M2) at (0.25, 0.5);
    \coordinate (M3) at (0.75, 0.5);
    \coordinate (M4) at (1, 1);
    \coordinate (M5) at (1.25, 0.5);
    
    \coordinate (C1) at (0.5, 0.375);
    \coordinate (C2) at (1, 0.625);

    \draw (A) -- (M1);
    \draw (A) -- (M2);
    \draw (B) -- (M1);
    \draw (B) -- (M5);
    \draw (C) -- (M5);
    \draw (C) -- (M4);
    \draw (D) -- (M4);
    \draw (D) -- (M2);

    \draw (C1) -- node[midway,sloped]{\tikz \draw[<-] (0,0) -- (0.3,0);} (C2);
    \draw (M2) -- node[midway,sloped]{\tikz \draw[->] (0,0) -- (0.3,0);} (C1);
    \draw (M1) -- node[midway,sloped]{\tikz \draw[->] (0,0) -- (0.3,0);} (C1);
    \draw (M5) -- node[midway,sloped]{\tikz \draw[<-] (0,0) -- (0.3,0);} (C2);
    \draw (M4) -- node[midway,sloped]{\tikz \draw[->] (0,0) -- (0.3,0);} (C2);

    % Draw the points
    \foreach \p in {C1,C2,M1,M2,M4,M5} {
        \fill[black] (\p) circle (0.75pt);
    }

    % place symbols at barycenters
    \node at (0.85, 0.25) {\scalebox{1.5}{$\circlearrowleft$}};
    \node at (0.275, 0.2) {\scalebox{1.5}{$\circlearrowleft$}};
    \node at (0.65, 0.75) {\scalebox{1.5}{$\circlearrowleft$}};
    \node at (1.185, 0.78) {\scalebox{1.5}{$\circlearrowleft$}};

\end{tikzpicture}
    \caption{Oriented dual mesh $\Dh$}
\end{subfigure}
\caption{2d example of the oriented dual mesh following \cite{Desbrun05, Hirani03}.}
\label{fig:dualcellmesh}
\end{figure}

In fact, a consequence of the construction of the dual mesh and the way we have defined orientation of polyhedral cells is the important identity:
\begin{equation}\label{partialstarsig}
\partial_{k}^{\smp}(*\sigma)=(-1)^{n-k+1} \sum_{\tau \in \mathcal{A}(\sigma)} *\tau+B_\sigma, \quad \forall \sig\in\del_{n-k},
\end{equation}
where $B_\sig$ consists of cells in $\delCB_{k-1}$. 

We can now define the subspace of $\C_{\smp, k}$ corresponding to $\delC_{*}$ as
\begin{equation*}
\C_{\smp, k}^*:= \left\{\, \sum_{ \sig \in \del_{n-k}} a_{\sig} (* \sig)\,:\,  a_\sig \in \mathbb{R} \,\right\}.
\end{equation*}
The dual space of $\C_{\smp, k}^*$ is the space $\C_{\smp}^{*,k}$ of cochains. Note that by Property 1 in Definition \ref{defdualmesh}, $\C_{\smp}^{*,k}$ is exactly the subspace of $\C_{\smp}^k$ with zero boundary conditions (since $\delC_* = \delCI$). That is, for any $\tau\in\delCB_{k}$ and $\mathsf{w}\in \C_{\smp}^{*,k}$, we have that $\bl \mathsf{w}, \tau\br = 0$. Thus, if one is only interested in $\C_{\smp}^{*,k}$, variations of the formula \eqref{partialstarsig} without the $B_\sig$ term can be used as a definition (e.g. \cite[Definition 3.6.8]{Hirani03}, \cite[Definition 2.3]{Schulz20}).

In particular, if we restrict the coboundary operator $\dd_{\smp}^k$ defined in the previous section to $\C_{\smp}^{*,k}$, the boundary term in the formula \eqref{partialstarsig} always vanishes and so we in fact have the complex with boundary conditions:
\begin{alignat}{4}\label{cochainbccomplex}
&\C_{\smp}^{*,0}
&&\stackrel{\dd_{\smp}^0}{\xrightarrow{\hspace*{0.5cm}}}\
\C_{\smp}^{*,1}
&&\stackrel{\dd_{\smp}^1}{\xrightarrow{\hspace*{0.5cm}}}\
\cdots
&&\stackrel{\dd_{\smp}^{n-2}}{\xrightarrow{\hspace*{0.5cm}}}\
\C_{\smp}^{*,n-1}
\stackrel{\dd_{\smp}^{n-1}}{\xrightarrow{\hspace*{0.5cm}}}\
\C_{\smp}^{*,n}.
\end{alignat}

\subsection{DEC and the numerical scheme}\label{sec:DEC_numerical_scheme}
We now explicitly define and list the standard DEC operators relevant to the DEC Scheme for the Hodge-Laplacian. For reference, the Hodge star and codifferential can be found in \cite{Desbrun05, Hirani03}.
\begin{definition}[Discrete exterior derivative on cochains]
    The discrete exterior derivative in DEC consists of the coboundary operators $\dd^k$ and $\dd_{\smp}^k$ defined in \eqref{defdd} and \eqref{defdd2}.
\end{definition}

\begin{definition}[Discrete hodge star on cochains]\label{DiscHodgeStar}
    The discrete Hodge star in DEC is the operator $\ccstar:\C^k \to \C_{\smp}^{*,n-k}$ such that 
\begin{equation}\label{DEChodgestar}
    \frac{1}{|*\sig|} \bl \ccstar \mathsf{w}, *\sig\br  = \frac{1}{|\sig|}\bl \mathsf{w}, \sig\br, \qquad \forall \sig\in\del_k.
\end{equation}
The inverse discrete Hodge star $\ccstar^{-1}:\C_{\smp}^{*,n-k}\to \C^k$ is such that
\begin{equation}\label{DEChodgestarinverse}
    \frac{1}{|\sig|} \bl \ccstar^{-1} \mathsf{w}, \sig\br  = \frac{1}{|*\sig|}\bl \mathsf{w}, *\sig\br, \qquad \forall \sig\in\del_k.
\end{equation}
\end{definition}

\begin{definition}[Discrete codifferential on cochains]\label{DiscCodiff}
    The discrete codifferential in DEC is $\ddelta_{k}: \C^k \to \C^{k-1}$ defined as 
    \begin{equation*}
        \ddelta_k:= (-1)^k \ccstar^{-1} \dd_{\smp}^{n-k} \ccstar.
    \end{equation*}
\end{definition}
We unpack the definition of $\ddelta_k$. For $\mathsf{w}\in\C^k$ and $\sig\in\del_{k-1}$, we have
\begin{alignat}{2}\label{deltah_effect}
    \nonumber \bl \ddelta_k \mathsf{w}, \sig\br &= (-1)^k \frac{|\sig|}{|*\sig|} \bl \ccstar\mathsf{w}, \partial_{n-k+1}^{\smp}(*\sig)\br  && \quad \text{by \eqref{DEChodgestarinverse}, \eqref{defdd2}}\\
    &= \frac{|\sig|}{|*\sig|} \sum_{\tau\in\mathcal{A}(\sig)} \bl \ccstar \mathsf{w}, *\tau\br  && \quad \text{by \eqref{partialstarsig}}\\
    \nonumber&= \frac{|\sig|}{|*\sig|} \sum_{\tau\in\mathcal{A}(\sig)} \frac{|*\tau|}{|\tau|}\bl \mathsf{w}, \tau\br   && \quad \text{by \eqref{DEChodgestar}},
\end{alignat}
where we used that $\ccstar \mathsf{w}\in\C_{\smp}^{*,n-k}$ so $\bl \ccstar \mathsf{w}, \tau'\br = 0$ for all $\tau'\subset\partial\Omega$.

\begin{definition}[Discrete Hodge-Laplacian on cochains]\label{DiscHodgeLap}
    We define the discrete Hodge-Laplacian $\dDelta^k : \C^k \rightarrow \C^k$ as
\begin{equation*}
    \dDelta^k :=  \dd^{k-1} \ddelta_k +\ddelta_{k+1} \dd^k.
\end{equation*}
\end{definition}

\begin{definition}[DEC inner product on cochains]
    We define the DEC inner product on $\C^k$ as
\begin{equation}\label{def:cochain_inner_product}
    (\mathsf{v}, \mathsf{w})_{\C^k}:= \sum_{\sig\in\del_k} a_\sig\bl \mathsf{v}, \sig\br\bl \mathsf{w}, \sig\br\qquad \text{where }a_\sig = \frac{|*\sig|}{|\sig|}.
\end{equation}
\end{definition}
We see that $\ddelta_{k+1}$ is indeed the adjoint of $\dd^{k}$ under this inner product. 
\begin{lemma}[Adjoint property with cochain inner product]
It holds, 
\begin{equation}\label{adjunction_cochain}
  ( \ddelta_{k+1} \mathsf{u}, \mathsf{v})_{\C^k} = ( \mathsf{u}, \dd^k \mathsf{v})_{\C^k}  \qquad \forall \mathsf{u} \in \C^{k+1}, \mathsf{v} \in \C^k.   
\end{equation}
\end{lemma}
\begin{proof}
We compute using the definition of the bilinear form: 
\begin{alignat*}{2}
   (\ddelta_{k+1} \mathsf{u}, v)_{\C^k}= & \sum_{ \sigma \in \Delta_k} a_\sigma \bl \ddelta_{k+1} \mathsf{u}, \sigma\br  \bl \mathsf{v}, \sigma \br \\
   =& \sum_{ \sigma \in \Delta_k} \sum_{\tau \in \mathcal{A} (\sigma)}   a_\tau \bl \mathsf{u}, \tau\br \bl \mathsf{v}, \sigma\br  && \quad \text{by }\eqref{deltah_effect}\\
   = & \sum_{\tau \in \Delta_{k+1}} \sum_{\sigma \in \mathcal{B}(\tau)}  a_\tau  \bl \mathsf{u}, \tau\br  \bl \mathsf{v}, \sigma\br  && \quad \text{by }\eqref{commutesum}.
\end{alignat*}
On the other hand, 
\begin{alignat*}{2}
     ( u, d^k v)_{\C^k}  =& \sum_{\tau \in \Delta_{k+1}}   a_\tau  \bl \mathsf{u}, \tau\br  \bl \dd^k \mathsf{v}, \tau\br \\
   =& \sum_{\tau \in \Delta_{k+1}}   a_\tau  \bl \mathsf{u}, \tau\br  \bl \mathsf{v}, \partial_{k+1} \tau\br  && \quad \text{by \eqref{defdd}}\\
   =& \sum_{\tau \in \Delta_{k+1}} \sum_{\sigma \in \mathcal{B}(\tau)}  a_\tau  \bl \mathsf{u}, \tau\br  \bl \mathsf{v}, \sigma\br  && \quad \text{by }\eqref{boundarydef}.
\end{alignat*}
\end{proof}

% We show this adjoint relationship explicitly for forms after we define the equivalent inner product on forms in Section \ref{sec:DEC_numerical_scheme_forms}.

\begin{definition}[Harmonic cochains]
    The space of harmonic cochains is
\begin{equation*}
    \C_{\Hh}^k:=\{\mathsf{w}\in \C^k: \dd \mathsf{w}=\mathsf{0}, \ddelta_k\mathsf{w} = \mathsf{0}\},
\end{equation*}
where $\mathsf{0}$ is the zero cochain.
\end{definition}

We let $P_{\C_{\Hh}^k} \mathsf{w}$ denote the $\C^k$-orthogonal projection of $\mathsf{w}$ onto $\C_{\Hh}^k$ satisfying
\begin{equation*}
    (P_{\C_{\Hh}^k} \mathsf{w}, \mathsf{v})_{\C^k} = (\mathsf{w}, \mathsf{v})_{\C^k}\quad \forall \mathsf{v}\in\C_{\Hh}^k.
\end{equation*}

Finally, the DEC numerical scheme approximating the Hodge-Laplacian problem \eqref{continuousstrongform} is: 
\begin{align}\label{DECscheme}
\nonumber \text{Find $\mathsf{u} \in$}&\text{ $\C^k$ which satisfies }\\
\dDelta^k \mathsf{u}&=  \mathsf{f} - P_{\C_{\Hh}^k}\mathsf{f},\\
\nonumber P_{\C_{\Hh}^k} \mathsf{u} &= \mathsf{0},
\end{align}
where $\mathsf{f}:=\dr^k f$ and $\dr^k$ is the natural mapping of smooth forms to $\C^k$ known as the de Rham map which we discuss in Section \ref{sec:formcochainequivalence}. We note in particular that the solution to the DEC scheme is a cochain, but there is a natural corresponding differential form obtained via the Whitney map which we also discuss in Section \ref{sec:formcochainequivalence}. Moreover, by our definition of the discrete Hodge star, we naturally satisfy the discrete analogue of the conditions \eqref{csfbc1} and \eqref{csfbc2}.

\section{Whitney form setting}\label{sec:WhitneyForms}
In this section, we make explicit the connection between DEC and Whitney forms. We place particular emphasis on the rewriting of the operators defined in the previous section in terms of Whitney forms. 
% We show the equivalences of our form-based operations to the cochain-based operations in Appendix \ref{equivalence}. 
This formulation is what enables the tools we introduce in Sections \ref{preliminarytools} and \ref{sec:StabilityErrorEstimates}.

\subsection{Whitney forms on the primal mesh}\label{sec:formcochainequivalence}
There is a natural relationship between cochains and the Whitney forms originally considered by Whitney in \cite{Whitney57}. We make note of this relationship in this section, and we perform all of our analysis in later sections working directly with Whitney forms rather than cochains.

The spaces of Whitney forms are the spaces $V_h^k \subset H \Lambda^k(\Omega)$ which form the discrete complex:
\begin{alignat}{4}\label{Whitneyformcomplex}
&\mathbb{R}
\stackrel{\subset}{\xrightarrow{\hspace*{0.5cm}}}\
V_h^0
&&\stackrel{d^0}{\xrightarrow{\hspace*{0.5cm}}}\
V_h^1
&&\stackrel{d^1}{\xrightarrow{\hspace*{0.5cm}}}\
\cdots
&&\stackrel{d^{n-2}}{\xrightarrow{\hspace*{0.5cm}}}\
V_h^{n-1}
\stackrel{d^{n-1}}{\xrightarrow{\hspace*{0.5cm}}}\
V_h^{n},
\end{alignat}
and each $\omega_h \in V_h^k$ is uniquely defined by the following degrees of freedom (dofs).
\begin{equation}\label{Vhdofs}
    \int_\sig \tr_\sig \omega_h  \qquad \forall \sig \in \del_k.
\end{equation}
The dual bases of these dofs are then the bases of the Whitney forms which we denote $\{\phi_{\sig}\}_{\sig\in \del_k}$.

We now note the connection between cochains and Whitney forms. To this end, we define the de Rham map $\dr^k$ on smooth enough forms to admit an $L^1$ trace on $k$ simplices by letting $\dr^k \omega\in \C^k$ be defined by
\begin{equation*}
    \bl \dr^k \omega, \sig\br = \int_{\sig} \tr_{\sig} \omega, \quad \forall \sig\in\del_k.
\end{equation*}
In particular, the de Rham map is well-defined over $V_h^k$. Next, we define the Whitney map $\wm^k:\C^k \to V_h^k$ by
\begin{equation*}
    \wm^k \mathsf{w} = \sum_{\sig\in\del_k}\bl \mathsf{w}, \sig\br\phi_\sig, \quad \mathsf{w}\in\C^k.
\end{equation*}

Notable properties of $\dr^k$ and $\wm^k$ are
\begin{subequations}
\begin{alignat}{2}
    \dr^k \wm^k \mathsf{w} &= \mathsf{w}, \quad &&\mathsf{w}\in\C^k, \label{RW=I}\\
    d^k \wm^k \mathsf{w} &=\wm^{k+1} \dd^k \mathsf{w}, \quad &&\mathsf{w}\in\C^k,\label{dW=Wd}\\
    \dd^k \dr^k \omega_h &= \dr^{k+1} d^k \omega_h \quad && \omega_h\in V_h^k,\label{dR=Rd}\\
    \wm^k \dr^k \omega_h &= \omega_h, \quad &&\omega_h\in V_h^k.\label{WR=I}
\end{alignat}
\end{subequations}
The first and second properties can be found in \cite[Chapter IV, Section 27]{Whitney57}. The third property is exactly Stokes' theorem while the fourth property readily follows from the definitions of $V_h^k$, $\dr^k$, and $\wm^k$.
We denote by $\Pi^k:=\wm^k\dr^k$, the canonical projection onto $V_h^k$ for smooth enough forms to admit an $L^1$ trace on $k$-simplices. That is, $\Pi^k$ is such that
\begin{alignat}{1}
    \int_{\sigma} \tr_\sigma \Pi^k \omega = \int_{\sigma} \tr_\sigma \omega \qquad \forall \sigma \in \Delta_k.\label{canonicalprojection}
\end{alignat}
From \eqref{dW=Wd} and \eqref{dR=Rd}, the canonical projection has the following commuting property:
\begin{equation}\label{commutePid}
  d^k \Pi^k= \Pi^{k+1} d^k. 
\end{equation}

\subsection{Generalized Whitney forms on the dual mesh}
To rewrite the DEC scheme in terms of forms, we will use generalizations of the Whitney forms from the previous section to cellular (in our case, polyhedral cell) complexes developed by Christiansen (\cite{Christiansen08, Christiansen09} and \cite[Section 5]{Christiansen11}). We detail the construction of these spaces in Appendix \ref{GenWhitneyConstruction}. 

\begin{proposition}\label{GeneralizedWhitneyForms}
There exists a sequence of spaces $W_h^k \subset H \Lambda^k(\Omega)$ which form the discrete complex:
\begin{alignat}{4}\label{cellcomplex}
&\mathbb{R}
\stackrel{\subset}{\xrightarrow{\hspace*{0.5cm}}}\
W_h^0
&&\stackrel{d^0}{\xrightarrow{\hspace*{0.5cm}}}\
W_h^1
&&\stackrel{d^1}{\xrightarrow{\hspace*{0.5cm}}}\
\cdots
&&\stackrel{d^{n-2}}{\xrightarrow{\hspace*{0.5cm}}}\
W_h^{n-1}
\stackrel{d^{n-1}}{\xrightarrow{\hspace*{0.5cm}}}\
W_h^{n}.
\end{alignat}
Moreover, each $\omega_h \in W_h^k$ is uniquely defined by the following dofs
\begin{equation}\label{Whdofs}
    \int_\tau \tr_\tau \omega_h  \qquad \forall \tau \in \Delta_k^{\smp}.
\end{equation}
The dual bases of these dofs are the generalized Whitney forms $\{\phi_{\tau}\}_{\tau\in \delC_k}$.
\end{proposition}
We define the spaces with boundary conditions as
\begin{equation*}
    \mathring{W}_h^{k} := W_h^k\cap \mathring{H}\La^k(\Omega).
\end{equation*}

Next, we define the dual de Rham map $\dr_{\smp}^k$ on smooth enough forms to admit an $L^1$ trace on the dual $k$-cells in $\delC_{*}$ by letting $\dr_{\smp}^k \omega\in \C_{\smp}^{*,k}$ be defined by
\begin{equation*}
    \bl \dr_{\smp}^k \omega, *\sig\br = \int_{*\sig} \tr_{*\sig} \omega, \quad \forall \sig\in\del_{n-k}.
\end{equation*}
We also define the dual Whitney map $\wm_{\smp}^k:\C_{\smp}^k \to W_h^k$ by
\begin{equation*}
    \wm_{\smp}^k \mathsf{w} = \sum_{\tau\in\delC_k}\bl \mathsf{w}, \tau\br\phi_\tau, \quad \mathsf{w}\in\C_{\smp}^k.
\end{equation*}
Note that $\wm_{\smp}^k$ maps $\C_{\smp}^{*,k}$ to $\mathring{W}_h^{k}$. We let $\Pi_{\smp}^k : = \wm_{\smp}^k \dr_{\smp}^k$ be the canonical projection to $\mathring{W}_h^{k}$. For a smooth enough $k$-form $u$, $\Pi_{\smp}^k u$ is such that
\begin{equation*}
\int_{\tau} \tr_\tau \Pi_{\smp}^k u= \int_{\tau} \tr_\tau u   \qquad \forall \tau \in \delC_{*,k}.   
\end{equation*}

Moreover, since the generalized Whitney forms form the complex \eqref{cellcomplex}, we see that $d^k \wm_{\smp}^k \mathsf{w} = \wm_{\smp}^{k+1} \dd_{\smp}^k \mathsf{w}$ for all $\mathsf{w}\in \C_{\smp}^{*, k}$, and from Stokes' theorem, $\dd_{\smp}^k \dr_{\smp}^k \omega_h = \dr_{\smp}^{k+1} d^k \omega_h$ for all smooth enough differential $k$-forms $\omega_h$. Hence, one readily sees that
\begin{equation}\label{Pidcommutedual}
  d^k \Pi_{\smp}^k = \Pi_{\smp}^{k+1} d^k.  
\end{equation}

\begin{remark}[Previous work involving generalized Whitney forms]
    Similar generalizations of Whitney forms also appear in the analysis of Virtual Element Methods (e.g. see \cite[Section 6]{daVeiga23}). Generalized Whitney forms on dual meshes have also been introduced in \cite{Gillette11}, but these are such that the dual Whitney map does not commute with the exterior derivative (cf. \cite[Remark 4.7]{Gillette11}). In particular, this means that the key property \eqref{Pidcommutedual} is not satisfied and so these constructions can not be immediately used for the framework developed in this work.
\end{remark}

\subsection{The DEC scheme on forms}\label{sec:DEC_numerical_scheme_forms}
We now turn to defining the equivalent operations on forms to the DEC operations introduced in Section \ref{sec:DEC_numerical_scheme}. Once these are defined, we can rewrite the DEC scheme purely in terms of differential forms.

We begin by defining bilinear forms that are inner-products on the spaces $V_h^k$ which are the natural corresponding inner products to \eqref{def:cochain_inner_product} as
\begin{alignat}{1}\label{iprelation}
    \llb v, w\rrb := ( \dr v , \dr w )_{\C^k} = \sum_{ \sigma \in \Delta_k} a_\sigma \int_\sigma \tr_\sig u  \int_\sigma \tr_\sig v \qquad \text{where } a_{\sig}= \frac{|*\sig|}{|\sig|}.
\end{alignat}
We denote the corresponding norm by $\dn{\cdot}$.
we note that for any smooth enough $k$-form $u$, 
\begin{equation}\label{Pinormequal}
    \dn{\Pi^k u} = \dn{u},
\end{equation}
which follows directly from the definitions of $\Pi^k$ and $\dn{\cdot}$.

We also define bilinear forms which are inner products on the spaces $\mathring{W}_h^{n-k}$.
\begin{equation*}
    \llb u, v \rrb_{*} = \sum_{\sig\in\del_{k}} b_\sig \int_{*\sig} \tr_{*\sig} u \int_{*\sig} \tr_{*\sig} v, \qquad \text{where } b_{\sig}= \frac{|\sig|}{|*\sig|}.
\end{equation*}
We denote the corresponding discrete norm by $\dn{\cdot}_{*}$.

It will also be useful to define the discrete $H\La^k(\Omega)$ inner product as:
\begin{equation*}
    \llb u, v\rrb_{H_d}:=\llb u, v\rrb + \llb du, dv\rrb,
\end{equation*}
with corresponding norm $\dn{\cdot}_{H_d}$.

\begin{definition}[Discrete Hodge star on forms]\label{FormDiscHodgeStar}
The discrete Hodge star operator $\star_h$ is such that for a sufficiently smooth $k$-form $\omega$, we let $\star_h \omega \in \mathring{W}_h^{n-k}$ be given by
\begin{equation}\label{FormHodgestar}
    \frac{1}{|*\sigma|} \int_{*\sig}  \tr_{*\sig}\star_h \omega=  \frac{1}{|\sigma|} \int_{\sigma}  \tr_{\sig} \omega \qquad \forall \sigma \in \Delta_k(\Th).
\end{equation}
Similarly, for a sufficiently smooth $n-k$ form, $\omega$, we define $\star_h^{-1} \omega \in V_h^k$ by
\begin{equation}\label{FormHodgestarinverse}
\frac{1}{|\sigma|} \int_{\sigma}\tr_{\sig}  \star_h^{-1} \omega=  \frac{1}{|*\sigma|} \int_{*\sigma}\tr_{*\sig}  \omega \qquad \forall \sigma \in \Delta_k(\Th).
\end{equation}
\end{definition}
We easily see that $\star_h^{-1} \star_h$ is the identity if we restrict to $V_h^k$, and $\star_h \star_h^{-1}$ is the identity if we restrict to $\mathring{W}_h^{k}$. 

\begin{definition}[Discrete Codifferential on forms]\label{FormDiscCodiff}
The operator $\delta_{h, k}: V_h^k \rightarrow V_h^{k-1}$  is defined as 
\begin{equation*}
    \delta_{h, k} =   (-1)^k \star_h^{-1} d^{n-k} \star_h. 
\end{equation*}
\end{definition}

\begin{definition}[Discrete Hodge-Laplacian on forms]\label{FormDiscHodgeLap}
We define the discrete Hodge-Laplacian $\Delta_h^k : V_h^k \rightarrow V_h^k$ as
\begin{equation*}
    \Delta_h^k :=  d^{k-1} \delta_{h,k} +\delta_{h,k+1} d^k.
\end{equation*}
\end{definition}

From the definitions above, one readily obtains the following relationships.
\begin{lemma}\label{lem:equivalences}
    For $\mathsf{w}\in\C^k$, it holds,
    \begin{subequations}
    \begin{alignat}{1}
        \dr^{k+1} d^k \wm^k \mathsf{w}&= \dd^k \mathsf{w}, \label{dequiv}\\
        \dr_{\smp}^{n-k} \star_h \wm^k \mathsf{w} &= \ccstar \mathsf{w}, \label{hsequiv}\\
        \dr^{k-1}\delta_{h,k}\wm^k \mathsf{w}&= \ddelta_k\mathsf{w}, \label{deltaequiv}\\
        \dr^k \Delta_h^k \wm^k  \mathsf{w}&= \dDelta^k \mathsf{w},\label{HLequiv}\\
        \dr^k P_{\Hh_h^k} \wm^k\mathsf{w} &= P_{\C_{\Hh}^k}\mathsf{w},\label{Hhequiv}
    \end{alignat}
    and for $\mathsf{w}\in\C_{\smp}^{*,k}$, it holds
    \begin{equation}
            \dr^{n-k} \star_h^{-1} \wm_{\smp}^k \mathsf{w}= \ccstar^{-1} \mathsf{w}.  \label{hsiequiv}
    \end{equation}
    \end{subequations}
\end{lemma}

We remark that from \eqref{RW=I} and \eqref{deltaequiv}, we have $\delta_{h,k}\,\delta_{h,k-1} = 0$, and so the discrete codifferential also forms a complex over the Whitney forms.
Moreover, using \eqref{RW=I}, \eqref{adjunction_cochain}, and \eqref{dR=Rd}, we immediately have the following lemma.

\begin{lemma}[Adjoint property for discrete codifferential on forms]
    For any $v\in V_h^k$ and $w\in V_h^{k-1}$, 
\begin{equation}\label{adjunction}
    \llb \delta_{h,k} v, w\rrb = \llb v, d^{k-1} w\rrb.
\end{equation}
\end{lemma}
Next, we define the space of discrete harmonic $k$-forms as 
\begin{equation*}
    \Hh_h^k:=\{v\in V_h^k: d^k v = 0, \delta_{h,k} v = 0\},
\end{equation*}
and let $P_{\Hh_h^k}$ denote the orthogonal projection onto $\Hh_h^k$ satisfying
\begin{equation*}
    \llb P_{\Hh_h^k} u, v \rrb = \llb u, v\rrb \quad \forall v\in \Hh_h^k.
\end{equation*}

We now write explicitly the DEC scheme from Section \ref{sec:DEC_numerical_scheme} on forms.
\begin{align}\label{numericalscheme}
\nonumber \text{Find $u_h$}& \text{$\in V_h^k$ which satisfies }\\
\Delta_h^k u_h&=  \Pi^k f - P_{\Hh_h^k}\Pi^k f,\\
\nonumber P_{\Hh_h^k}u_h &= 0.
\end{align}

Letting $\rho_h=\delta_{h,k} u_h$ and $p_h= P_{\Hh_h^k}f$, we can rewrite \eqref{numericalscheme} as
\begin{subequations}\label{mixedform}
  \begin{align}
    \text{Find }(u_h,\rho_h)\in V_h^k\times V_h^{k-1}
      \text{ su} &\text{ch that}\nonumber\\
    \rho_h - \delta_{h,k}\,u_h                          &= 0,          \label{mf1}\\
    d^{k-1}\rho_h + \delta_{h,k+1} d^k u_h + p_h        &= \Pi^k f,    \label{mf2}\\
    P_{\Hh_h^k}u_h                                       &= 0.          \label{mf3}
  \end{align}
\end{subequations}

Finally, we note the following result on the equivalence between the solution one obtains from the DEC scheme over cochains \eqref{DECscheme} and that of the DEC scheme over forms \eqref{numericalscheme}.
\begin{proposition}
    Let $\mathsf{u}$ be the solution to \eqref{DECscheme} and $u_h$ be the solution to \eqref{numericalscheme}. Then, $\mathsf{u} = \dr^k u_h$.
\end{proposition}
\begin{proof}
    We have from \eqref{RW=I} and \eqref{Hhequiv} that
    \begin{alignat*}{1}
        \dr^k f  - P_{\C_{\Hh}^k} \dr^k f &= \dr^k \wm^k\dr^k f - \dr^k P_{\Hh_h^k} \wm^k \dr^k f = \dr^k\left( \Pi^k f - P_{\Hh_h^k} \Pi^k f\right).
    \end{alignat*}
    Then, using the DEC scheme over forms \eqref{numericalscheme} and \eqref{HLequiv}, we have that
    \begin{equation*}
        \dr^k\left( \Pi^k f - P_{\Hh_h^k} \Pi^k f\right) = \dr^k \Delta_h^k u_h = \dr^k \Delta_h^k \wm^k\dr^k u_h = \dDelta^k \dr^k u_h.
    \end{equation*}
    However, we know that $\dDelta^k \mathsf{u} = \dr^k f  - P_{\C_{\Hh}^k} \dr^k f$ and so the result then follows by uniqueness of the solution to \eqref{DECscheme}.
\end{proof}
Hence, we can analyze the DEC scheme \eqref{DECscheme} by instead analyzing \eqref{numericalscheme}. We remark that one can analyze the DEC scheme \eqref{DECscheme} directly without ever invoking generalized Whitney forms (see Remark \ref{rem:cochain_analysis} below) but we choose to analyze \eqref{numericalscheme} to highlight the naturality of the tools we use from the FEEC perspective.

\section{Preliminary tools and results}\label{preliminarytools}

In this section, we introduce crucial tools and note results which will be useful for the proof of our error estimates in Section \ref{sec:StabilityErrorEstimates}. For brevity, we will no longer explicitly note operators in terms of $k$ and $n$ nor write $\tr$ when integrating over a $k$-cell or $k$-face when this is clear from context.

\subsection{Interpolant that commutes with the discrete codifferential}\label{sec:projections}
We begin by introducing an interpolant that sends smooth enough $k$ forms to $V_h^k$ as
\begin{alignat*}{1}
    J:= \star_h^{-1} \Pi_{\smp} \star.
\end{alignat*}
The key property of this interpolant is that it satisfies a commuting property with the discrete co-differential $\delta_h$. We state and prove this result in the following lemma.
\begin{lemma}[Commuting property of $J$]
For $u$, a smooth enough $k$-form for $Ju$ and $J\delta u$ to be well-defined, 
\begin{equation}\label{commuteJdelta}
\delta_h J u = J \delta u.   
\end{equation}
\end{lemma}
\begin{proof}
We directly compute using the definitions of $\delta_h$, $J$, and $\delta$:
  \begin{alignat*}{2}
    \delta_h J u= & (-1)^k \star_h^{-1} d \star_h J u\\
      =&  (-1)^k \star_h^{-1} d  \Pi_{\smp} \star u  && \quad \text{since }\star_h \star_h^{-1} \text{ is identity}\\
      = & (-1)^k \star_h^{-1} \Pi_{\smp} d \star u && \quad \text{by \eqref{Pidcommutedual}} \\
      = &  (-1)^k \star_h^{-1} \Pi_{\smp} \star  \star^{-1} d \star u &&\quad \text{since }\star \star^{-1} \text{ is identity} \\
      = & J \delta u. 
\end{alignat*}  
\end{proof}
We will also use the following lemma in the proof of the error estimates.
\begin{lemma}
For $u$, a smooth enough $k$-form for $J$ to be well-defined, 
\begin{equation}\label{Jbound}
     \dn{J u} = \dn{\star u}_*.
\end{equation}
\end{lemma}
\begin{proof}
    We compute using the definitions of $J$, $\star_h^{-1}$, and $\Pi_{\smp}$. 
\begin{align*}
    \dn{Ju}^2 &= \sum_{ \sigma \in \Delta_k} a_\sig \left(\int_\sigma J u\right)^2 = \sum_{ \sigma \in \Delta_k} a_\sig \left(\int_\sigma \star_h^{-1} \Pi_{\smp} \star u\right)^2\\
    &= \sum_{ \sigma \in \Delta_k} b_\sig \left(\int_{* \sigma } \Pi_{\smp} \star u\right)^2 = \sum_{ \sigma \in \Delta_k}  b_\sig \left(\int_{* \sigma}  \star u\right)^2 = \dn{\star u}_*^2
\end{align*}
% \begin{alignat*}{1}
%    \dn{Ju}^2 = \llb J u, Ju\rrb =&   \sum_{ \sigma \in \Delta_k} a_\sig \int_\sigma J u  \int_\sigma J u \\
%    =&  \sum_{ \sigma \in \Delta_k} a_\sig \int_\sigma \star_h^{-1} \Pi_{\smp} \star u  \int_\sigma \star_h^{-1} \Pi_{\smp} \star u \\
%    =&  \sum_{ \sigma \in \Delta_k} b_\sig \int_{* \sigma } \Pi_{\smp} \star u  \int_{*\sigma}  \Pi_{\smp} \star u\\
%    =&  \sum_{ \sigma \in \Delta_k}  b_\sig \int_{* \sigma}  \star u  \int_{*\sigma} \star u = \dn{\star u}_*^2.
% \end{alignat*}
\end{proof}

\subsection{The difference $\Pi-J$}
The idea to use the operator $J$ is inspired by the analysis in \cite[Section 3]{Schulz20} where the authors estimate the difference $\ccstar \dr - \dr_{\smp} \star$ in appropriate cochain norms. Our analysis will instead rely on estimating the difference $\Pi - J = \wm (\dr -  \ccstar^{-1} \dr_{\smp} \star)$ in the discrete norm.

To this end, we aim to characterize the kernel of the difference $\Pi-J$. In particular, we argue that as long as our dual mesh satisfies the properties outlined in section \ref{sec:dual_mesh} for all $\sigma \in \Delta_k$, then the kernel of $\Pi-J$ contains certain piecewise constant $k$-forms. 

We define an $n$-dimensional ``diamond cell" associated to any $\sigma \in \del_k$ denoted $\dc(\sig)$. The construction of the diamond cells parallels that of the construction of the dual cells detailed earlier in Section \ref{sec:dual_mesh}. 
\begin{definition}[Diamond cell]\label{defdc}
Let $\sig\in\del_k$. For notational convenience, we will denote $\sig= \sig_k$. Then, there exists an ordered set of simplices $\{\sig_{0}, \sig_{1}, \dots, \sig_{k-1}, \sig,  \sig_{k+1}, \dots, \sig_{n}\}$ such that $\sig_{j}\subset \sig_{j+1}$ for all $j<n$, and $\sig_{j}\in \del_{j}$ for each $j\in\{0,\dots,n\}$. The collection of all such ordered sets of simplices is denoted $\mathcal{S}(\sig)$. For any $S=\{\sig_0, \dots, \sig_n\}\in\mathcal{S}(\sig)$, we can identify an $n$-simplex $\tau_S:=[c(\sig_0), \dots, c(\sig_n)]$ where $c(\sig)$ denotes the circumcenter of $\sig$. Then, we define
\begin{equation*}
    \dc(\sig):= \bigcup_{S\in\mathcal{S}(\sig)} \tau_S.
\end{equation*}
\end{definition}
We will need the following assumption on the diamond cells.
\begin{assumption}[Poincar\'e inequality on $\dc(\sig)$]\label{PoincareAssumption}
    We assume the simplicial decomposition $\{\tau_S\}_{S\in \mathcal{S}(\sigma)}$ of $\dc(\sigma)$ noted in Definition \ref{defdc} is shape-regular. Moreover, we further assume that for any $\sig\in\del_k$ and $v\in H^1(\Omega)$ with average zero over $\dc(\sig)$, we have
    \begin{equation}
        \|v\|_{L^2(\dc(\sig))} \leq C_P h |v|_{H^1(\dc(\sig))},
    \end{equation}
    where $C_P$ is independent of $h$ and is uniformly bounded over all $\sig\in \del_k$.
\end{assumption}
There are various ways one can establish such a result. For instance, if it is the case that all of the cells $\dc(\sig)$ are convex, then one can apply the results of \cite{Bebendorf, Payne60}. More generally, the boundedness of $C_P$ on nonconvex domains can be found in \cite[Lemma 3.7]{Eymard00}.

The notion of the diamond cell has previously arisen in the finite volume literature \cite[Section 2.2]{Coudiere11} and in the gradient discretization method \cite[Definition 7.2]{Droniou18} in dimensions 2 and 3, although the definitions are such that various points can be used rather than specifically circumcenters.

We define the cellular mesh of diamond cells as
\begin{equation}
\mathcal{Z}_h^k:=\{ \dc(\sigma): \sigma \in \Delta_k \}.
\end{equation}
An example is illustrated in Figure \ref{fig:diamondcellmesh}.

\begin{figure}[H]
\centering
\begin{subfigure}[b]{0.3\textwidth}
\centering
\begin{tikzpicture}[scale = 2.5, >={Latex[round]}] % Thicker lines

    \coordinate (A) at (0, 0);
    \coordinate (B) at (1, 0);
    \coordinate (C) at (1.5, 1);
    \coordinate (D) at (0.5, 1);

    \coordinate (C1) at (0.5, 0.375);
    \coordinate (C2) at (1, 0.625);

    \draw [blue] (A) -- (B);
    \draw [blue] (B) -- (D);
    \draw [blue] (D) -- (A);
    
    \draw [blue] (B) -- (C);
    \draw [blue] (C) -- (D);
    
    % Draw the points
    \foreach \p in {A,B,C,D} {
        \fill[black] (\p) circle (0.75pt);
    }

\end{tikzpicture}
\caption{Primal mesh $\Th$}
\end{subfigure}
\hfill
\begin{subfigure}[b]{0.3\textwidth}
\centering

\begin{tikzpicture}[scale = 2.5]

    \coordinate (A) at (0, 0);
    \coordinate (B) at (1, 0);
    \coordinate (C) at (1.5, 1);
    \coordinate (D) at (0.5, 1);
    
    \coordinate (M1) at (0.5, 0);
    \coordinate (M2) at (0.25, 0.5);
    \coordinate (M3) at (0.75, 0.5);
    \coordinate (M4) at (1, 1);
    \coordinate (M5) at (1.25, 0.5);
    
    \coordinate (C1) at (0.5, 0.375);
    \coordinate (C2) at (1, 0.625);

    \draw (A) -- (B);
    \draw (B) -- (C);
    \draw (C) -- (D);
    \draw (D) -- (A);

    \draw [red](M2) -- (C1);
    \draw [red](M1) -- (C1);
    \draw [red](M3) -- (C1);
    \draw [red](M5) -- (C2);
    \draw [red](M4) -- (C2);
    \draw [red](M3) -- (C2);
    
    % Draw the points
    \foreach \p in {A,B,C,D,C1,C2,M1,M2,M4,M5} {
        \fill[black] (\p) circle (0.75pt);
    }
    
\end{tikzpicture}
    \caption{Dual mesh $\Dh$}
\end{subfigure}
\hfill
\begin{subfigure}[b]{0.3\textwidth}
\centering

\begin{tikzpicture}[scale = 2.5]

    \coordinate (A) at (0, 0);
    \coordinate (B) at (1, 0);
    \coordinate (C) at (1.5, 1);
    \coordinate (D) at (0.5, 1);
    
    \coordinate (M1) at (0.5, 0);
    \coordinate (M2) at (0.25, 0.5);
    \coordinate (M3) at (0.75, 0.5);
    \coordinate (M4) at (1, 1);
    \coordinate (M5) at (1.25, 0.5);
    
    \coordinate (C1) at (0.5, 0.375);
    \coordinate (C2) at (1, 0.625);

    % --- Added shading here ---
    % Lightly shade the central diamond cell (D-C1-B-C2)
    \fill[gray, opacity=0.3] (D) -- (C1) -- (B) -- (C2) -- cycle;
    % --------------------------

    \draw [thick, black](A) -- (B);
    \draw [thick, black](B) -- (C);
    \draw [thick, black](C) -- (D);
    \draw [thick, black](D) -- (A);

    \draw [blue, opacity = 0.5](A) -- (B);
    \draw [blue, opacity = 0.5](B) -- (C);
    \draw [blue, opacity = 0.5](C) -- (D);
    \draw [blue, opacity = 0.5](D) -- (A);
    \draw [blue, opacity = 0.5](D) -- (B);

    \draw [thick](A) -- (C1);
    \draw [thick](B) -- (C1);
    \draw [thick](B) -- (C2);
    \draw [thick](C) -- (C2);
    \draw [thick](D) -- (C2);
    \draw [thick](D) -- (C1);

    \draw [red, opacity = 0.5](M2) -- (C1);
    \draw [red, opacity = 0.5](M1) -- (C1);
    \draw [red, opacity = 0.5](M3) -- (C1);
    \draw [red, opacity = 0.5](M5) -- (C2);
    \draw [red, opacity = 0.5](M4) -- (C2);
    \draw [red, opacity = 0.5](M3) -- (C2);
    
    % Draw the points
    \foreach \p in {A,B,C,D,C1,C2} {
        \fill[black] (\p) circle (0.75pt);
    }
    
\end{tikzpicture}
    \caption{Diamond cell mesh, $\mathcal{Z}_h^1$}
\end{subfigure}
\caption{2d illustration of a primal mesh, dual mesh, and diamond cell mesh for $k=1$. The primal edges are highlighted in blue. The dual edges are highlighted in red. A single diamond cell is shaded in gray in (c). Each diamond cell contains exactly one primal edge and its dual.}
\label{fig:diamondcellmesh}
\end{figure}

We then see that $\mathcal{Z}_h^k$ is indeed a conforming cellular mesh as it partitions $\Omega$ into sets such that their interiors do not intersect. That is, 
\begin{equation}
\overline{\Omega}= \bigcup_{S \in \mathcal{Z}_h^k} S,
\end{equation}
and if $S, T \in  \mathcal{Z}_h^k$ then $\text{int}(S) \cap \text{int}(T) = \emptyset$.

We then define the space of piecewise polynomial $k$-forms on the diamond cell mesh $\mathcal{Z}_h^k$ as
\begin{equation*}
    \pol_p\La^k(\mathcal{Z}_h^k): = \{u \in L^2 \La^k(\Omega): u|_{S} \in \pol_p \La^k(\R^n),  \forall S \in \mathcal{Z}_h^k\},
\end{equation*}
where $\pol_p\La^k(\R^n)$ is the space of $k$-forms on $\R^n$ with polynomial coefficients of degree less than or equal to $p$. Notice that $\Pi$ (resp. $J$) is well-defined over $\pol_p\La^k(\mathcal{Z}_h^k)$ because each diamond cell contains exactly one $\sig\in\del_k$ (resp. $*\sig\in\delC_k$). 

We now begin to characterize the kernel of $\Pi-J$ with the following lemma. 
\begin{lemma}\label{lem:kernel}
If
\begin{equation}\label{piecewiseconstantrelation}
    \frac{1}{|\sigma|}\int_{\sigma} u = \frac{1}{|*\sigma|}\int_{*\sigma}\star u \quad \forall \sigma \in \Delta_k,
\end{equation}
then $(\Pi - J) u=0$.
\end{lemma}

\begin{proof}
It suffices to show that  $\int_{\sigma} (\Pi - J)u=0$ for all $\sigma \in \Delta_k$. To this end, 
\begin{equation*}
    \int_{\sigma} (\Pi - J)u = \int_{\sigma} u - \int_{\sigma}\star_h^{-1}\Pi_{\smp}\star u = \int_{\sigma} u -\frac{|\sigma|}{|*\sigma|}\int_{*\sigma}\Pi_{\smp}\star u = \int_{\sigma} u - \frac{|\sigma|}{|*\sigma|}\int_{*\sigma}\star u = 0.
\end{equation*}
\end{proof}
Note that the condition \eqref{piecewiseconstantrelation} is equivalent to $\Pi_{\smp} \star u = \star_h u$. Now, the following lemma crucially uses that elements of $\del$ and their duals are orthogonal as noted in Property 2 of Definition \ref{defdualmesh}. This idea was first used in the proof of Theorem 3.2 in \cite{Schulz20}.
\begin{lemma}\label{kernel}
We have that $\pol_0\Lambda^k (\mathcal{Z}_h^k) \subset \Kern(\Pi - J)$.
\end{lemma}
\begin{proof}
Let $\sig\in \Delta_k$ and $u\in \pol_0\Lambda^k (\mathcal{Z}_h^k)$.  By Property 4 in Definition \ref{defdualmesh}, there exists an orthonormal basis of $\R^n$ of vectors $t_1, \ldots, t_n$ such that $t_1, \dots, t_k$ span $P(\sig)$ and encode the orientation of $\sig$, and $t_{k+1}, \dots, t_n$ span $P(*\sig)$ and encode the orientation of $*\sig$ such that $\det[t_1, \dots, t_n]=1$ or equivalently 
\begin{equation}\label{wedges1}
    dy_1 \wedge\cdots\wedge dy_n = \vol,
\end{equation}
where $dy_1, \dots, dy_n$ is the dual basis to $t_1, \dots, t_n$. Moreover, since $t_1, \dots, t_k$ are orthonormal vectors which orient $\sig$ (similarly for $t_{k+1}, \dots, t_n$ and $*\sig$), 
\begin{equation}\label{wedges0}
    dy_1\wedge \cdots\wedge dy_k =  \vol_{\sig}, \quad dy_{k+1}\wedge\cdots\wedge dy_n = \vol_{*\sig},
\end{equation}
where $\vol_{\tau}$ indicates the volume form on $\tau$. 

Since $u\in \pol_0\Lambda^k (\mathcal{Z}_h^k)$, $u$ is constant on $\dc(\sigma)$. Hence, we can represent $u$ on $\dc(\sigma)$ as
\begin{equation*}
    u= \sum_{ I\in\Sigma(k,n)} a_{I} dy_I  \quad \text{ on } \dc(\sigma),
\end{equation*}
where $a_I \in \R$. Let $S=(1, \ldots, k)$. Then if $I\in\Sigma(k,n)$ such that $I \neq S$, it must be that there exists $\ell\in I$ such that $t_\ell$ is a vector that is orthogonal to $P(\sigma)$. Hence,    
\begin{equation*}
  \int_{\sigma} dy_I= 0 \quad \forall I\in\Sigma(k,n), \, I \neq S.  
\end{equation*}
Now, by \eqref{wedges0}, $dy_{S} = \vol_\sig$ so 
\begin{equation*}
  \frac{1}{|\sigma|}\int_{\sigma} u= \frac{a_S}{|\sigma|} \int_{\sigma} dy_{S} = a_S.
\end{equation*}
Now, for each $I\in\Sigma(k,n)$, let $I^*$ be such that $dy_{I^*} =\star dy_I$. In particular, this means that $S^* = (k+1, \dots, n)$. One can see this considering $\omega= dy_{S}$ in \eqref{orientationrelation}. Then, we have
\begin{equation*}
    \llangle \star dy_{S}, dy_{I} \rrangle \vol = dy_{S}\wedge dy_{I} = 0 \quad  \forall I\in\Sigma(n-k,n),\, I\neq(k+1, \dots, n),
\end{equation*}
and
\begin{equation*}
    \llangle \star dy_{S}, dy_{(k+1, \dots, n)} \rrangle \vol = dy_{S}\wedge dy_{(k+1, \dots, n)} = \vol,
\end{equation*}
by \eqref{wedges1}. Thus, we similarly have
\begin{equation*}
    \int_{*\sigma} dy_{I^*}=0 \quad  \quad \forall I\in\Sigma(k,n),\, I \neq S,
\end{equation*}
since when $I^*\neq S^*$, there exists $\ell\in I^*$ such that $t_\ell$ is orthogonal to $*\sigma$. 
Now, since $S^* = (k+1, \dots, n)$, we have $dy_{S^{*}} = \vol_{*\sig}$ by \eqref{wedges0} so
\begin{equation*}
  \frac{1}{|*\sigma|}\int_{*\sigma} \star u=\frac{a_{S}}{|*\sigma|} \int_{*\sigma} dy_{S^*}=a_S.
\end{equation*}
\end{proof}

Now, we define $\Pi_{\mathcal{Z}} v $ to be the $L^2\La^k$-orthogonal projection of $v$ onto $\pol_0\La^k(\mathcal{Z}_h^k)$ satisfying
\begin{equation*}
    ( \Pi_\mathcal{Z} u - u, q)_{\Omega} = 0, \quad \forall q\in\pol_0\La^k(\mathcal{Z}_h^k).
\end{equation*}
Note that for every $\sig\in\del_k$, $\Pi_\mathcal{Z} u|_{\dc(\sig)}$ is exactly the average of $u$ over $\dc(\sig)$.

Now, a consequence of Lemma \ref{kernel} is the following lemma.
\begin{lemma}\label{Pi-Jbounddn}
    For any smooth enough $k$-form $v$,
    \begin{equation}\label{Pi-Jbounddnineq}
        \dn{(\Pi-J) v}^2 \leq 2\left(\dn{v-\Pi_{\mathcal{Z}}v}^2 + \dn{\star(v-\Pi_{\mathcal{Z}}v)}_*^2\right).
    \end{equation}
\end{lemma}
\begin{proof}
By Lemma \ref{kernel},
    \begin{alignat*}{1}
    \dn{(\Pi-J) v} &= \dn{(\Pi-J) (v-\Pi_{\mathcal{Z}}v)}\\
    &\leq \dn{\Pi (v-\Pi_{\mathcal{Z}}v)} + \dn{J (v-\Pi_{\mathcal{Z}}v)}\\
    &= \dn{v-\Pi_{\mathcal{Z}}v} + \dn{\star(v-\Pi_{\mathcal{Z}}v)}_*,
\end{alignat*}
where we used \eqref{Jbound} and \eqref{Pinormequal}. Squaring both sides and applying Young's inequality gives \eqref{Pi-Jbounddnineq}.
\end{proof}

We will also need the following scaled trace inequalities for smooth forms.
\begin{lemma}
    There exists a constant, $C_{\tr}$ independent of $h$ such that for any $\sig\in\del_k$, we have
    \begin{subequations}
    \begin{align}
        \|\tr_\sig u\|_{L^2(\sig)}^2&\leq C_{\tr}\left(h^{-(n-k)}\|u\|_{L^2(\dc(\sig))}^2 + \sum_{s=1}^{r_1}h^{2s-(n-k)}|u|_{H^s(\dc(\sig))}^2\right),\label{trace_ineq}\\
        \|\tr_{*\sig} v\|_{L^2(*\sig)}^2&\leq C_{\tr}\left(h^{-k}\|v\|_{L^2(\dc(\sig))}^2 + \sum_{s=1}^{r_2}h^{2s-k}|v|_{H^s(\dc(\sig))}^2\right),\label{trace_ineq2}
    \end{align}
    \end{subequations}
    where $u$ is a smooth enough $k$-form, $v$ is a smooth enough $(n-k)$-form, $r_1 := \lceil\frac{n-k}{2}+\varepsilon\rceil$, and $r_2 := \lceil\frac{k}{2}+\varepsilon\rceil$ for any $0<\varepsilon < 1$.
\end{lemma}
\begin{proof}
    Now, we can write $\dc(\sig) = \cup_{i=1}^m \tau_i$ where each $\tau_i$ is one of the $\tau_S$ from the definition of the diamond cell $\dc(\sig)$. Then, we can define $\{\sig_j\}_{j=1}^\ell$ to be the set of intersections of $\tau_i$ with $\sig$. Let $\tau \in \{\tau_i\}_{i=1}^m$ and the intersection of its boundary with $\sig$ be denoted $\sig'$. Let $\widehat{\sig}$ denote the standard reference $k$-simplex, and $\widehat{\tau}$ denote the reference $n$-simplex.
    
    Next, it follows from iterated use of the trace theorem (e.g. \cite[Theorem 3.10]{ErnGuermondI21}) that for any $\widehat{w}\in H^{r}\La^k(\widehat{\tau})$,
    \begin{equation}\label{trace_ineq_reference}
        \|\tr_{\widehat{\sig}} \widehat{w}\|_{L^2(\widehat{\sig})}^2 \leq \widehat{C}\|\widehat{w}\|_{H^{r}(\widehat{\tau})}^2.
    \end{equation}
    Note that the result holds for $r=\frac{n-k}{2}+\varepsilon$, but we can take $r = \lceil\frac{n-k}{2}+\varepsilon\rceil$ due to the embedding of $H^{\lceil s \rceil}$ into $H^s$. While sharper estimates are possible, we do not investigate this here.
    
    Now, we let $\Phi$ be the invertible affine map which maps $\tau$ to $\widehat{\tau}$ and define $\widehat{u}=(\Phi^{-1})^{*}u$ where $\cdot^*$ denotes the pullback. Notice that $\tr_{\widehat{\sig}} \widehat{u} = (\Phi^{-1})^{*} \tr_{\sig} u$ (i.e. the trace commutes with $(\Phi^{-1})^{*}$). Moreover, by scaling we have that
    \begin{subequations}
        \begin{alignat}{1}
            \|\tr_\sig u\|_{L^2(\sig')} &\leq C_1 h^{-\frac{k}{2}}\|\tr_{\widehat{\sig}} \widehat{u} \|_{L^2(\widehat{\sig})},\label{scaling1}\\
            |\widehat{u}|_{H^s(\widehat{\tau})}&\leq C_2 h^{s+k-\frac{n}{2}} |u|_{H^s(\tau)},\label{scaling2}
        \end{alignat}
    \end{subequations}
    where $C_1, C_2$ are independent of $h$. Thus, using \eqref{scaling1}, \eqref{trace_ineq_reference}, and \eqref{scaling2}, we get
    \begin{align*}
        \|\tr_\sig u \|_{L^2(\sig')}^2 &\leq C_1 h^{-k}\|\tr_{\widehat{\sig}} \widehat{u} \|_{L^2(\widehat{\sig})}^2\leq C_1\widehat{C}h^{-k}\|\widehat{u}\|_{H^{r}(\widehat{\tau})}^2\\
        &\leq C_1\widehat{C}C_2 \left(h^{-(n-k)}\|u\|_{L^2(\tau)}^2 + \sum_{s=1}^{r}h^{2s-(n-k)}|u|_{H^s(\tau)}^2\right).
    \end{align*}
    % \begin{align*}
    %     \|\tr_\sig u \|_{L^2(\sig')}^2 &\leq C_1 h^{-k}\|\tr_{\widehat{\sig}} \widehat{u} \|_{L^2(\widehat{\sig})}^2\\
    %     &\leq C_1\widehat{C}h^{-k}\|\widehat{u}\|_{H^{r}(\widehat{\tau})}^2\\
    %     &\leq C_1\widehat{C}C_2 \left(h^{-(n-k)}\|u\|_{L^2(\tau)}^2 + \sum_{s=1}^{r}h^{2s-(n-k)}|u|_{H^s(\tau)}^2\right).
    % \end{align*}
Since we have the above inequality for all $\tau$ and corresponding $\sigma'$, we in fact have
\begin{equation*}
    \sum_{j=1}^\ell\|\tr_\sig u \|_{L^2(\sig_j)}^2 \leq C\sum_{i=1}^m\left(h^{-(n-k)}\|u\|_{L^2(\tau_i)}^2 + \sum_{s=1}^{r}h^{2s-(n-k)}|u|_{H^s(\tau_i)}^2\right),
\end{equation*}
where we used shape-regularity to pull out the constant from the sum. This shows \eqref{trace_ineq}.

Following the same procedure, but now using $\{\sig_j\}$ to be the set of intersections of the $\tau_i$ with $*\sig$ yields \eqref{trace_ineq2}. Note in particular that each $\sig_j$ will be an $n-k$ simplex in this case.
\end{proof}

Finally, we have the following result on the boundedness of the difference $\Pi-J$ in the discrete norm. 
\begin{lemma}\label{Pi-Jboundlemma}
    Let $v$ be a smooth enough $k$-form for $\Pi-J$ to be well-defined. Then, there exists a constant, $C$ depending on the constant $C_P$ from Assumption \ref{PoincareAssumption}, but independent of $h$ such that 
    \begin{equation}\label{Pi-Jbound}
        \dn{(\Pi-J) v}^2 \leq C\sum_{s=1}^{r} h^{2s}|v|_{H^s(\Omega)}^2,
    \end{equation}
    where $r = \max\left(\lceil\frac{n-k}{2}+\varepsilon\rceil, \lceil\frac{k}{2}+\varepsilon\rceil\right)$ for any $0<\varepsilon < 1$. 
\end{lemma}
\begin{proof}
    By Lemma \ref{Pi-Jbounddn}, it suffices to estimate $\dn{v-\Pi_{\mathcal{Z}}v}^2$ and $\dn{\star(v-\Pi_{\mathcal{Z}}v)}_*^2$. One can readily verify that there exists $C_*$ independent of $h$ such that $|*\sig|\leq C_{*} h^{n-k}$ for all $\sig\in\del_k$.
    Using this fact and the trace inequality \eqref{trace_ineq}, letting $w=v-\Pi_{\mathcal{Z}}v$, we have
    \begin{alignat*}{1}
        \dn{w}^2 &= \sum_{\sig\in\del_k} a_\sig \left(\int_\sig w\right)^2 \leq \sum_{\sig\in\del_k} |*\sig| \|w\|_{L^2(\sig)}^2\\
        &\leq \sum_{\sig\in\del_k} C_* h^{n-k} C_{\tr}\left(h^{-(n-k)}\|w\|_{L^2(\dc(\sig))}^2 + \sum_{s=1}^{r_1}h^{2s-(n-k)}|w|_{H^s(\dc(\sig))}^2\right)\\
        &= \sum_{\sig\in\del_k}C\left(\|w\|_{L^2(\dc(\sig))}^2 + \sum_{s=1}^{r_1} h^{2s}|w|_{H^{s}(\dc(\sig))}^2\right),
    \end{alignat*}
    where $r_1 := \lceil\frac{n-k}{2}+\varepsilon\rceil$ for any $0<\varepsilon < 1$.
    
    Similarly, one can readily verify that there exists $C'$ independent of $h$ such that $|\sig|\leq C' h^{k}$ for all $\sig\in\del_k$. Using this fact and the trace inequality \eqref{trace_ineq2}, we have 
    \begin{alignat*}{1}
        \dn{\star w}_*^2 &= \sum_{\sig\in\del_k} b_\sig \left(\int_{*\sig} \star w\right)^2 \leq \sum_{\sig\in\del_k} |\sig| \|\star w\|_{L^2(*\sig)}^2\\
        &\leq \sum_{\sig\in\del_k} C' h^{k} C_{\tr}\left(h^{-k}\|\star w\|_{L^2(\dc(\sig))}^2 + \sum_{s=1}^{r_2}h^{2s-k}|\star w|_{H^s(\dc(\sig))}^2\right)\\
        &= \sum_{\sig\in\del_k}C\left(\|w\|_{L^2(\dc(\sig))}^2 + \sum_{s=1}^{r_2} h^{2s}|w|_{H^{s}(\dc(\sig))}^2\right),
    \end{alignat*}
    where $r_2 := \lceil\frac{k}{2}+\varepsilon\rceil$ for any $0<\varepsilon < 1$. We used the isometric property of the Hodge star in the final equality .
    
    Note that $v-\Pi_\mathcal{Z}v$ has the same regularity as $v$ on the diamond cells so the above computations are justified as long as $v$ is regular enough. Moreover, in both computations, we have that the constant is uniformly bounded by shape-regularity and so can be pulled out of the sum.

    Finally, by Assumption \ref{PoincareAssumption}, since $v-\Pi_\mathcal{Z}v$ has average zero over $\dc(\sig)$,
    \begin{equation*}
        \|v-\Pi_{\mathcal{Z}}v\|_{L^2(\dc(\sig))}\leq C_P h |v|_{H^1(\dc(\sig))}.
    \end{equation*}
    Thus, the $L^2(\dc(\sig))$ term on the right hand side of the previous computation can be absorbed into the $H^1(\dc(\sig))$ term, and since we assume the Poincar\'e constant is uniformly bounded, we can pull it out of the sum to arrive at the result.
\end{proof}

\section{Error Estimates}\label{sec:StabilityErrorEstimates}
In this section, we state and prove our main results on stability and error estimates for the DEC scheme. In particular, we present our main result in the theorem below.
\begin{theorem}\label{mainresult}
    Let $e_u := \Pi u - u_h$, $e_{\rho}:= \Pi \rho - \rho_h$, and $e_p := p-p_h$ where $\Pi$ is the canonical projection satisfying \eqref{canonicalprojection}, $u_h$, $\rho_h$ are the solutions to \eqref{mixedform}, $p_h=P_{\Hh_h} f$, and $u,\rho$ are the solutions to \eqref{continuousmixedform} with $p = P_{\Hh}f$. Then, there exists $C$ independent of $h$ such that
    \begin{equation}\label{eq:mainresult}
        \dn{e_u}^2 + \dn{de_u}^2 + \dn{e_\rho}^2 + \dn{d e_\rho}^2 + \dn{e_p}^2 \leq C \mathcal{E}(u),
    \end{equation}
    where
    \begin{equation*}
        \mathcal{E}(u) = \dn{(J-\Pi) u}^2 + \dn{(\Pi-J)d u}^2 + \dn{(\Pi-J) \delta u}^2+\dn{(J-\Pi) \delta d u}^2 +\dn{(\Pi-J)u_\delta}^2,
    \end{equation*}
    and $u_\delta$ is as in \eqref{hodgedecomposition}.
\end{theorem}

One can see right away that by applying Lemma \ref{Pi-Jboundlemma} to Theorem \ref{mainresult} above, we obtain the following explicit error estimates.
\begin{corollary}[Error estimates]\label{errorestimate}
    Let $e_u$, $e_{\rho}$, and $e_p$ as in Theorem \ref{mainresult}. Then, we have the estimates
    \begin{alignat*}{1}
        \dn{e_u}^2 + \dn{d e_u}^2+ &\dn{e_{\rho}}^2+ \dn{d e_\rho}^2 + \dn{e_p}^2 \\ 
        &\leq  C\sum_{s=1}^{r} h^{2s}\Big(|u|_{H^s(\Omega)}^2 + |du|_{H^s(\Omega)}^2 + |\delta u|_{H^s(\Omega)}^2 + |\delta du|_{H^s(\Omega)}^2 + |u_\delta|_{H^s(\Omega)}^2 \Big),
    \end{alignat*}
    where $r = \max\left(\lceil\frac{n-k}{2}+\varepsilon\rceil, \lceil\frac{k}{2}+\varepsilon\rceil\right)$ for any $0<\varepsilon < 1$ and $C$ is independent of $h$. 
\end{corollary}
The estimates in Corollary \ref{errorestimate} are new for $k\geq 1$. In particular, Corollary \ref{errorestimate} implies that the errors converge at a rate of $\mathcal{O}(h)$. 

\subsection{Error equations}
In order to prove Theorem \ref{mainresult} we first need to write the error equations which requires the next crucial lemma.  
\begin{lemma}
 It holds,  
\begin{subequations}
\begin{alignat}{1}
\Pi \rho-\delta_h \Pi u&= T_1, \label{aux1} \\
   d \Pi \rho+  \delta_h d \Pi u +\Pi p &=\Pi f + T_2, \label{aux2}
\end{alignat}
\end{subequations}
where 
\begin{equation*}
    T_1 := (\Pi-J) \delta u + \delta_h  (J-\Pi) u,  \quad T_2 := (J-\Pi)  \delta d u + \delta_h  (\Pi-J)d  u.
\end{equation*}

% \begin{alignat*}{1}
% T_1&= (\Pi-J) \delta u + \delta_h  (J-\Pi) u,  \\
%  T_2&=   (J-\Pi)  \delta d u+ \delta_h  (\Pi-J)d  u.
% \end{alignat*}   
\end{lemma}

\begin{proof}
    
We start with \eqref{aux1}. We take the canonical projection of \eqref{cmf1} to get that
\begin{alignat*}{1}
0&=\Pi \rho-\Pi \delta u \\
&= \Pi \rho-J \delta u -(\Pi-J) \delta u \\
&= \Pi \rho- \delta_h  J u -(\Pi-J) \delta u \\
&= \Pi \rho- \delta_h  \Pi u -(\Pi-J) \delta u - \delta_h  (J-\Pi) u,
\end{alignat*}
where we used \eqref{commuteJdelta}.

Now to \eqref{aux2}. We take the canonical projection of \eqref{cmf2} to get that
\begin{alignat*}{1}
\Pi f &= \Pi d \rho+ \Pi \delta d u + \Pi p  \\
&=d \Pi \rho + J \delta d u + \Pi p + (\Pi-J)  \delta d u \\
&=d \Pi \rho + \delta_h J d u + \Pi p + (\Pi-J)  \delta d u \\
&=d \Pi \rho + \delta_h  d \Pi u + \Pi p + (\Pi-J)  \delta d u + \delta_h(J- \Pi)du,
\end{alignat*}
where we used \eqref{commutePid} and \eqref{commuteJdelta}.
\end{proof}

Recall we denote the errors as $e_{\rho}=\Pi\rho -\rho_h$, $e_u= \Pi u -u_h$, and $e_p= \Pi p - p_h$. Then, subtracting \eqref{mf1} from \eqref{aux1} and \eqref{mf2} from \eqref{aux2}, we can see that the errors satisfy
\begin{subequations}
\begin{alignat}{1}
e_{\rho}-\delta_h e_u&= T_1, \label{err812} \\
   d e_{\rho}+  \delta_h d e_u + e_p&=  T_2. \label{err813}
\end{alignat}
\end{subequations}
Moreover, using \eqref{mf3}, we can see that 
\begin{equation*}
    0 = P_{\Hh_h^k} u_h = P_{\Hh_h^k}(u_h - \Pi u) + P_{\Hh_h^k} \Pi u,\\
\end{equation*}
and so we also have that $e_u$ satisfies
\begin{equation}\label{err814}
    P_{\Hh_h^k} e_u = T_3, \quad \text{where } T_3 = P_{\Hh_h^k} \Pi u.
\end{equation}

Applying the DEC inner product with test forms, $\tau_h$, $v_h$, and $q_h$, and using \eqref{adjunction} yields
\begin{subequations}\label{errorequations}
\begin{alignat}{1}
    \llb e_{\rho}, \tau_h\rrb-  \llb e_u, d \tau_h\rrb &= \llb T_1, \tau_h\rrb, \qquad \forall \tau_h \in V_h^{k-1},\label{erroreq1}\\
    \llb d e_{\rho}, v_h\rrb + \llb d  e_u, d v_h\rrb + \llb e_p, v_h\rrb &=  \llb T_2, v_h\rrb, \qquad \forall  v_h \in V_h^{k},\label{erroreq2}\\
    \llb e_u, q_h \rrb &= \llb T_3, q_h\rrb, \qquad \forall q_h\in \Hh_h^k. \label{erroreq3}
\end{alignat}    
\end{subequations}

In particular, we arrive at the error equations
\begin{equation}\label{errorequations2}
    B((e_\rho, e_u, e_p), (\tau, v, q))= \llb T_1, \tau\rrb + \llb T_2, v\rrb + \llb T_3, q_h\rrb.
\end{equation}
where 
\begin{equation*}
    B((\rho, u, p),(\tau, v, q)) := \llb \rho, \tau\rrb -\llb u, d\tau\rrb + \llb d\rho, v\rrb +\llb du, dv\rrb + \llb p, v\rrb +\llb u, q\rrb.
\end{equation*}
\subsection{Discrete Stability}
We now provide a stability result for the DEC scheme.
\begin{theorem}[Discrete stability]\label{discstab}
    For any $(\rho, u, p) \in V_h^{k-1} \times V_h^k\times \Hh_h^k$, there exists $(\tau, v, q)\in V_h^{k-1}\times V_h^k\times \Hh_h^k$ such that
    \begin{subequations}
        \begin{alignat}{1}
            B((\rho, u, p),(\tau, v, q))&\geq \gamma \left(\dn{\rho}_{H_d}^2 + \dn{u}_{H_d}^2 + \dn{p}^2\right),\label{disc_stab1}\\
            \dn{\tau}_{H_d}+\dn{v}_{H_d}+ \dn{q}&\leq M\left(\dn{\rho}_{H_d}+\dn{u}_{H_d} + \dn{p}\right),\label{disc_stab2}
        \end{alignat}
    \end{subequations}
    where the constants $\gamma > 0$, and $M<\infty$ are independent of $h$.  
\end{theorem}

We prove Theorem \ref{discstab} generally following \cite[Section 7]{Arnold2006} with some modifications accounting for the DEC scheme. In particular, the key tool is a discrete Poincar\'e inequality (using the discrete norms).

We begin by noting that the discrete norms are norm equivalent to the $L^2(\Omega)$ norm on the Whitney forms under Assumption \ref{assump:UWC}. We state this in the following lemma without proof as the result follows from a straightforward scaling argument.
\begin{lemma}[Norm equivalence]
There exists constants $c_{\mathsf{ne}}$ and $C_{\mathsf{ne}}$ independent of $h$ such that for any $u_h\in V_h^k$, 
    \begin{equation}\label{normequivalence}
        c_{\mathsf{ne}}^{-1}\dn{u_h}\leq \|u_h\|_{L^2(\Omega)} \leq C_{\mathsf{ne}}\dn{u_h}.
    \end{equation}
\end{lemma}

We now define the discrete spaces:
\begin{equation*}
    \Zz V_h^k:=\{u\in V_h^k: du = 0\},
\end{equation*}
and
\begin{equation*}
    \Zz V_h^{k,\perp}:=\{u\in V_h^k: \llb u, v\rrb = 0, \quad \forall v\in \Zz V_h^k\}.
\end{equation*}
Note that the space $\Zz V_h^{k,\perp}$ differs from the standard orthogonal complement space in FEEC (e.g. \cite[Section 5.6]{Arnold2006}) as the orthogonality is now with respect to the inner product $\llb\cdot,\cdot\rrb$ rather than the $L^2\La^k$-inner product.

Let $\overline{\Pi}$ be a $H\La^k(\Omega)$-bounded commuting projection (see for example \cite{Arnold21, ChristiansenWinther08, Ern2025, Falk2014}) satisfying the following for all  $z\in H\La^k(\Omega)$.
\begin{subequations}
    \begin{alignat}{2}
        \|\overline{\Pi}z\|_{H_d}&\leq C_{\overline{\Pi}}\|z\|_{H_d},\label{L2bddproj}\\
        d\overline{\Pi}z &= \overline{\Pi}dz. \label{Projcommute}
    \end{alignat}
\end{subequations}

Also, recall the Poincar\'e inequality \cite[Theorem 2.2]{Arnold2006}:
\begin{equation}\label{continuousPoinc}
    \|u\|\leq \tilde{C}_P \|du\| \qquad \forall u\in (\Zz H\La^k)^{\perp}.
\end{equation}

We now have the following discrete Poincar\'e inequality.
\begin{lemma}[Discrete Poincar\'e Inequality]
    For any $u\in \Zz V_h^{k,\perp}$, it holds,
    \begin{equation}\label{eq:discretePoinc}
        \dn{u}_{H_d} \leq C_{P,h}\dn{du},
    \end{equation}
    where $C_{P,h}=c_\mathsf{ne} C_{\overline{\Pi}} \tilde{C}_P C_\mathsf{ne}$.
\end{lemma}
\begin{proof}
    Let $u\in \Zz V_h^{k,\perp}$. Let $z$ be the $L^2$-projection of $u$ onto the space $(\Zz H\La^k)^{\perp}$. Then, $z\in (\Zz H\La^k)^{\perp}$ is such that $dz = du$. Hence, by \eqref{Projcommute}, $d \overline{\Pi} z=\overline{\Pi} dz=du$. Therefore, $u-\overline{\Pi} z \in \Zz V_h^k$ and, by our hypothesis of $u$, $\llb u, u-\overline{\Pi} z\rrb=0$ and in fact $\llb u, u-\overline{\Pi} z\rrb_{H_d}=0$.    

Now, we can see that
\begin{alignat*}{1}
    \dn{u}_{H_d}^2 = \llb u, u\rrb_{H_d} = \llb u, u-\overline{\Pi}z\rrb_{H_d} + \llb u, \overline{\Pi}z\rrb_{H_d} = \llb u, \overline{\Pi}z\rrb_{H_d} \leq \dn{u}_{H_d}\dn{\overline{\Pi}z}_{H_d}.
\end{alignat*}

Thus, using the norm equivalence \eqref{normequivalence} and \eqref{L2bddproj}, we have that
\begin{alignat*}{1}
    \dn{u}_{H_d}\leq \dn{\overline{\Pi}z}_{H_d}\leq c_{\mathsf{ne}} \|\overline{\Pi}z\|_{H_d} \leq c_{\mathsf{ne}} C_{\overline{\Pi}}\|z\|_{H_d}  \leq c_{\mathsf{ne}} C_{\overline{\Pi}} \tilde{C}_P\|dz\|_{L^2(\Omega)},
\end{alignat*}
and we conclude by noting that
\begin{alignat*}{1}
\|dz\|_{L^2(\Omega)} =\|du\|_{L^2(\Omega)} \leq C_{\mathsf{ne}}\dn{du}.
\end{alignat*}

\end{proof}

Next, we recall the following spaces:
\begin{equation*}
    \Bh_h^k:=d V_h^{k-1}, \qquad \Hh_h^k=\{w\in V_h^k: dw = 0, \llb w,dv\rrb = 0, \: \forall v\in V_h^{k-1}\}.
\end{equation*}

Then, we have the following discrete Hodge decomposition.
\begin{lemma}[Discrete Hodge Decomposition]
    Let $u\in V_h^k$. Then, we can decompose $u$ as
    \begin{equation*}
        u = \phi + \psi + \theta
    \end{equation*}
    where $\phi\in \Bh_h^k$, $\psi\in \Hh_h^k$, and $\theta\in \Zz V_h^{k,\perp}$. 
\end{lemma}
\begin{proof}
    First, since $\Zz V_h^k$ is a closed subspace of $V_h^k$, we can write $V_h^k = \Zz V_h^k \oplus \Zz V_h^{k,\perp}$. Similarly, we can write $\Zz V_h^k = \Bh_h^k \oplus (\Bh_h^k)^\perp$ since $\Bh_h^k$ is a closed subspace of $\Zz V_h^k$ where $(\Bh_h^k)^\perp$  is the orthogonal complement of $\Bh_h^k$ with respect to the DEC inner product. Finally, one readily sees that $(\Bh_h^k)^\perp = \Hh_h^k$ by the definition.
\end{proof}
 
Finally, we can prove the stability result. The proof follows exactly that of \cite[Section 7.5]{Arnold2006}, but now with the DEC inner products and norms.

\begin{proof}[Proof of Theorem \ref{discstab}]
    We begin by applying the discrete Hodge decomposition to $u\in V_h^k$ to write $u = \phi + \psi + \theta$ for $\phi \in \Bh_h^k$, $\psi\in \Hh_h^k$, and $\theta\in \Zz V_h^{k,\perp}$. In particular, note that $\phi$, $\psi$, and  $\theta$ are pairwise orthogonal under the DEC inner product. Then, $\dn{u}^2 = \dn{\phi}^2+ \dn{\psi}^2+\dn{\theta}^2$. Moreover, $\phi = d\mu$ for some $\mu\in \Zz V_h^{k-1, \perp}$ and $d\theta = du$. Thus, by the discrete Poincar\'e inequality \eqref{eq:discretePoinc}, we can see that
    \begin{equation}\label{aux:Poincares}
        \dn{\mu}_{H_d}\leq C_{P,h} \dn{\phi}, \qquad \dn{\theta}_{H_d}\leq C_{P,h}\dn{du}.
    \end{equation}
    Now, considering $\tau = \rho - t\mu$ for $t:=\frac{1}{C_{P,h}^2}$, $v=u+d\rho+p$, and $q=\psi-p$ we can see that
    \begin{alignat*}{1}
        B((\rho, u, p),(\tau, v, q)) &= \dn{\rho}^2-t\llb \rho, \mu\rrb +t\llb u, d\mu\rrb +\dn{d\rho}^2 + \dn{du}^2 +\dn{p}^2 + \llb u, \psi\rrb \\
        &=\dn{\rho}^2 +\dn{d\rho}^2 + \dn{du}^2 + \dn{p}^2 - t\llb \rho, \mu\rrb + t\llb u, \phi\rrb + \dn{\psi}^2\\
        &=\dn{\rho}^2 +\dn{d\rho}^2 + \dn{du}^2 + \dn{p}^2 - t\llb \rho, \mu\rrb + t\dn{\phi}^2 + \dn{\psi}^2,
    \end{alignat*}
    where we used that $\theta$ and $\phi$ are orthogonal under the DEC inner product in the last equality. Now, notice that by Young's inequality and \eqref{aux:Poincares},
    \begin{alignat*}{1}
        - t\llb \rho, \mu\rrb & \geq -\frac{1}{2}\dn{\rho}^2 -\frac{1}{2}t^2\dn{\mu}^2 \geq -\frac{1}{2}\dn{\rho}^2 -\frac{1}{2}t^2 C_{P,h}^2\dn{\phi}^2,
    \end{alignat*}
    and so we can estimate again using \eqref{aux:Poincares},
    \begin{alignat*}{1}
        B((\rho, u,p),(\tau, v,q)) &\geq \frac{1}{2}\dn{\rho}^2 +\dn{d\rho}^2 + \dn{du}^2 + \dn{p}^2 + \dn{\psi}^2 + \left(t-t^2\frac{C_{P,h}^2}{2}\right)\dn{\phi}^2\\
        &\geq \frac{1}{2}\dn{\rho}^2 +\dn{d\rho}^2 + \frac{1}{2}\dn{du}^2 + \dn{p}^2 + \dn{\psi}^2 + \frac{1}{2C_{P,h}^2}\dn{\phi}^2 + \frac{1}{2C_{P,h}^2}\dn{\theta}^2\\
        &\geq \frac{1}{2}\dn{\rho}^2 +\dn{d\rho}^2 + \frac{1}{2}\dn{du}^2 + \dn{p}^2 + \tilde{C}\dn{u}^2,
    \end{alignat*}
    where $\tilde{C} = \min\left(1, \frac{1}{2 C_{P,h}^2}\right)$. Thus, we obtain \eqref{disc_stab1} with $\gamma$ only depending on $C_{P,h}$. Now, we can estimate
\begin{alignat*}{1}
    \dn{\tau}_{H_d}+\dn{v}_{H_d} + \dn{q} &\leq \dn{\rho}_{H_d}+t\dn{\mu}_{H_d} + \dn{u}_{H_d} + \dn{d\rho}+ 2\dn{p}+\dn{\psi}\\
    &\leq \sqrt{2}\dn{\rho}_{H_d} + \frac{1}{C_{P,h}}\dn{\phi}+\dn{u}_{H_d}+ 2\dn{p}+\dn{\psi}.
\end{alignat*}
Thus, noting that $\dn{\phi}\leq\dn{u}$ and $\dn{\psi}\leq\dn{u}$, we readily obtain \eqref{disc_stab2} with $M$ only dependent on $C_{P,h}$.
    
\end{proof}

\subsection{Proof of Theorem \ref{mainresult}}
The error equations \eqref{errorequations}, along with Theorem \ref{discstab} allow us to now prove Theorem \ref{mainresult}.
\begin{proof}
Using \eqref{disc_stab1} and \eqref{disc_stab2}, we can see that there exists $\tau\in V_h^{k-1}$, $v\in V_h^k$, and $q\in \Hh_h^k$ such that
\begin{alignat*}{1}
    B((e_\rho, e_u, e_p), (\tau, v, q)) &\geq \gamma \left(\dn{e_\rho}_{H_d}^2 + \dn{e_u}_{H_d}^2 + \dn{e_p}^2\right)\\ 
    &\geq \frac{\gamma}{2 M} \left(\dn{e_\rho}_{H_d} + \dn{e_u}_{H_d}+\dn{e_p}\right) \left(\dn{\tau}_{H_d} + \dn{v}_{H_d}+\dn{q}\right). 
\end{alignat*}

Hence, using \eqref{errorequations2}, we have that
\begin{equation}\label{aux1001}
    \dn{e_\rho}_{H_d} + \dn{e_u}_{H_d} + \dn{e_p} \leq \frac{2M}{\gamma} \left(\sup_{\tau_h\in V_{h}^{k-1}}\frac{\llb T_1, \tau_h\rrb}{\dn{\tau_h}_{H_d}} + \sup_{v_h\in V_h^k}\frac{\llb T_2, v_h\rrb}{\dn{v_h}_{H_d}} + \sup_{q_h\in \Hh_h^k}\frac{\llb T_3, q_h\rrb}{\dn{q_h}}\right).
\end{equation}

Now, using \eqref{adjunction}, and the Cauchy-Schwarz inequality, we can see that for any $\tau_h\in V_h^{k-1}$,
\begin{alignat*}{1}
    \llb T_1, \tau_h\rrb &= \llb (\Pi-J)\delta u, \tau_h\rrb + \llb \delta_h(J-\Pi) u, \tau_h\rrb = \llb (\Pi-J)\delta u, \tau_h\rrb + \llb (J-\Pi) u, d\tau_h\rrb\\
    &\leq \dn{ (\Pi-J)\delta u} \dn{\tau_h} + \dn{(J-\Pi) u}\dn{d\tau_h},
\end{alignat*}
and for any $v_h\in V_h^k$,
\begin{alignat*}{1}
    \llb T_2, v_h\rrb &= \llb (J-\Pi)\delta d u, v_h\rrb + \llb \delta_h(\Pi-J)du, v_h\rrb = \llb (J-\Pi)\delta d u, v_h\rrb + \llb (\Pi-J)du, dv_h\rrb\\
    &\leq \dn{(J-\Pi)\delta d u}\dn{v_h}+\dn{(\Pi-J)du}\dn{dv_h}.
\end{alignat*}
Next, note that by \eqref{continuousstrongform}, we know $u\perp \Hh^k$. Hence, by \eqref{hodgedecomposition}, we can write $u = u_d+u_\delta$ where $u_d = d\mu$ for some $\mu\in H\La^{k-1}$ and $u_\delta = \delta \eta$ for some $\eta \in \mathring{H}_\delta\La^{k+1}$. Then, $T_3 = P_{\Hh_h^k} \Pi u = P_{\Hh_h^k}\Pi u_{d} + P_{\Hh_h^k}\Pi u_\delta$, and we can compute using \eqref{adjunction} for any $q_h\in\Hh_h^k$,
\begin{equation*}
    \llb P_{\Hh_h^k}\Pi u_{d}, q_h\rrb = \llb \Pi u_d, q_h\rrb = \llb d\Pi \mu, q_h\rrb = \llb \Pi\mu, \delta_h q_h\rrb = 0.
\end{equation*}
On the other hand, we can see that using \eqref{commuteJdelta}
\begin{equation*}
    \Pi u_\delta = \Pi \delta \eta = (\Pi-J)\delta\eta + J\delta\eta = (\Pi-J)u_\delta + \delta_h J \eta,
\end{equation*}
and so using \eqref{adjunction}, we have that
\begin{alignat*}{1}
    \llb P_{\Hh_h^k}\Pi u_{\delta}, q_h\rrb = \llb \Pi u_\delta, q_h\rrb = \llb(\Pi-J)u_\delta, q_h\rrb + \llb \delta_h J\eta, q_h\rrb& = \llb(\Pi-J)u_\delta, q_h\rrb + \llb J\eta, dq_h\rrb\\
    &= \llb(\Pi-J)u_\delta, q_h\rrb.
\end{alignat*}

Hence, for any $q_h\in\Hh_h^k$, we can see that
\begin{equation*}
    \llb T_3, q_h\rrb = \llb(\Pi-J)u_\delta, q_h\rrb\leq \dn{(\Pi-J)u_\delta} \dn{q_h}.
\end{equation*}

Thus, applying these three bounds to \eqref{aux1001} implies that
\begin{alignat*}{1}
    \dn{e_\rho}_{H_d} + \dn{e_u}_{H_d} +\dn{e_p} \leq \frac{2M}{\gamma}\Big(\dn{ (\Pi-J)\delta u}& + \dn{(J-\Pi) u} + \dn{(J-\Pi)\delta d u}\\
    &+ \dn{(\Pi-J)du} + \dn{(\Pi-J)u_\delta}\Big),
\end{alignat*}
which in turn implies \eqref{eq:mainresult}.
\end{proof}

\begin{remark}\label{rem:cochain_analysis}
    In fact, the entirety of the analysis of DEC which we have presented can be conducted only considering cochains and directly using the DEC scheme on cochains presented in \eqref{DECscheme}. In particular, this means that the generalized Whitney forms are not necessary for the proof of convergence of DEC. We briefly give the idea of this analysis. First, one can define $\mathcal{J}:=\ccstar^{-1} \dr \star$. Then, rather than estimating $\dn{\Pi-J}$, one would estimate $\dr-\mathcal{J}$ in the norm induced by the cochain inner product \eqref{def:cochain_inner_product}. The kernel of $\dr-\mathcal{J}$ is naturally the same as the kernel of $\Pi-J$ and so all of the same estimates hold. Hence, one can simply repeat all the above analyses by replacing the appropriate objects with their cochain counterparts ($\Pi$ becomes $\dr$, $J$ becomes $\mathcal{J}$, $\dn{\cdot}$ becomes the norm induced by the cochain inner product, etc.). 
\end{remark}

\section{Superconvergence}\label{sec:superconvergence}
In our numerical experiments, we notice higher convergence rates on certain symmetric meshes and manufactured solutions. In an effort to explain this phenomena, we show the idea of how one can obtain superconvergent error estimates in our framework under specific mesh conditions and assumptions on the solution.

 We begin by defining the subset of diamond cells corresponding to boundary simplices as
\begin{equation*}
    \overline{\mathcal{Z}}_h^k:= \{\dc(\sig):\sig\in\del_k, \sig\subset\partial\Omega\}.
\end{equation*}
Then, we define 
\begin{equation*}
    \overline{\pol}_1 \La^k (\mathcal{Z}_h^k):= \{u\in \pol_1\La^k(\mathcal{Z}_h^k): u|_S \equiv 0, \forall S\in \overline{\mathcal{Z}}_h^k\},
\end{equation*}
and let $\overline{\Pi}_{\mathcal{Z}} v $ be the $L^2\La^k$ projection of a $k$-form $v$ onto $\overline{\pol}_1\La^k(\mathcal{Z}_h^k)$. 

We also define the space:
\begin{equation*}
    \mathring{H}^2\La^k(\Omega):=\left\{\sum_{I\in\Sigma(k,n)}a_Idx_I: a_I\in \mathring{H}^2(\Omega)\right\},
\end{equation*}
where $\mathring{H}^2(\Omega)$ is the space of $H^2(\Omega)$-regular functions such that the function and its first derivatives vanish on $\partial\Omega$.

Then, if $v\in \mathring{H}^2\La^k(\Omega)$, one has 
\begin{equation*}
    \|v-\overline{\Pi}_{\mathcal{Z}} v\|_{L^2(\dc(\sig))} + h|v-\overline{\Pi}_{\mathcal{Z}} v|_{H^1(\dc(\sig))}  = \|v\|_{L^2(\dc(\sig))} + h|v|_{H^1(\dc(\sig))} \leq Ch^2\|v\|_{H^2(\dc(\sig))},
\end{equation*}
for all $\sig\in\del_k$ such that $\sig\subset\partial\Omega$, and one can use a Bramble-Hilbert argument on the other diamond cells to establish
\begin{equation}\label{BHsuper}
    \|v-\overline{\Pi}_{\mathcal{Z}}v\|_{L^2(\Omega)} + h|v-\overline{\Pi}_{\mathcal{Z}}v|_{H^1(\Omega)} \leq Ch^2|v|_{H^2(\Omega)}, \quad v\in \mathring{H}^2\La^k(\Omega).
\end{equation}

Now, the key idea is that if $\overline{\pol}_1\La^k(\mathcal{Z}_h^k)\subset \Kern(\Pi-J)$, then we have that
    \begin{alignat*}{1}
    \dn{(\Pi-J) v} &= \dn{(\Pi-J) (v-\overline{\Pi}_{\mathcal{Z}}v)}\leq \dn{v-\overline{\Pi}_{\mathcal{Z}}v} + \dn{\star(v-\overline{\Pi}_{\mathcal{Z}}v)}_*,
\end{alignat*}
where we used \eqref{Jbound} and \eqref{Pinormequal}. Therefore, if we follow the same proof as Lemma \ref{Pi-Jbound}, but now apply \eqref{BHsuper} instead of the Poincar\'e inequality at the last step, we obtain
\begin{equation}\label{Pi-Jboundsuper}
    \dn{(\Pi-J) v}^2 \leq C\sum_{s=2}^{r} h^{2s}|v|_{H^s(\Omega)}^2, \quad v\in \mathring{H}^2\La^k(\Omega).
\end{equation}
where $r = \max\left(\lceil\frac{n-k}{2}+\varepsilon\rceil, \lceil\frac{k}{2}+\varepsilon\rceil\right)$ for any $0<\varepsilon < 1$.

Then, simply applying this result to Theorem \ref{mainresult} yields the following.
\begin{corollary}[Superconvergent Error Estimates]\label{corollarysuper}
    If $\overline{\pol}_1\La^k(\mathcal{Z}_h^k)\subset \Kern(\Pi-J)$ and $u,du,\delta u,\delta d u, u_\delta\in \mathring{H}^2\La^k(\Omega)$, then    
    \begin{alignat*}{1}
        \dn{e_u}^2 + \dn{d e_u}^2+ &\dn{e_{\rho}}^2+ \dn{d e_\rho}^2 + \dn{e_p}^2 \\ &\leq  C\sum_{s=2}^{r} h^{2s}\Big(|u|_{H^s(\Omega)}^2 + |du|_{H^s(\Omega)}^2 + |\delta u|_{H^s(\Omega)}^2 + |\delta du|_{H^s(\Omega)}^2 + |u_\delta|_{H^s(\Omega)}^2 \Big),
    \end{alignat*}
    where $r = \max\left(\lceil\frac{n-k}{2}+\varepsilon\rceil, \lceil\frac{k}{2}+\varepsilon\rceil\right)$ for any $0<\varepsilon < 1$ and $C$ is independent of $h$. 
\end{corollary}

We now give geometric conditions for $\overline{\pol}_1\La^k(\mathcal{Z}_h^k)\subset \Kern(\Pi-J)$ to hold. Let $\sig\in\del_k$. Then, we let $t_1, \ldots, t_n$ be the orthonormal basis of $\R^n$ introduced in the proof of Lemma \ref{lem:kernel} and $dy_1, \dots, dy_n$ be the dual basis of $t_1, \dots, t_n$. 
Now, letting $\bld{z} = (z_1,\dots,z_n)$ be the vector of coordinates corresponding to $e_1,\dots, e_n$, then the coordinates corresponding to $t_1, \dots, t_n$ are $y_1, \dots, y_n$ where $y_i= \bld{z}\cdot t_i$, and letting $O$ be a point of intersection between $P(\sig)$ and $P(*\sig)$, we can write $P(\sig)=\{O+\sum_{i=1}^k a_i t_i, a_i\in\R \}$ and $P(*\sig)=\{O+\sum_{i=k+1}^n b_i t_i, b_i\in\R \}$. Without loss of generality, we suppose $O$ is the origin. In particular, we will also use that $y_i|_{P(\sig)}=0$ for all $k+1\leq i\leq n$ and $y_i|_{P(*\sig)}=0$ for all $1\leq i\leq k$. 

The centroid of $\sig$, denoted $\mathfrak{c}(\sig)$, is such that $\mathfrak{c}(\sig) = \sum_{i=1}^k \mathfrak{c}(\sig)_i t_i$ where 
\begin{equation}\label{centroidsig}
    \mathfrak{c}(\sig)_i = \frac{1}{|\sig|}\int_{\sig} y_i \vol_\sig.
\end{equation}
Similarly, the centroid of $*\sig$, denoted $\mathfrak{c}(*\sig)$, is such that $\mathfrak{c}(*\sig) = \sum_{i=k+1}^n \mathfrak{c}(*\sig)_i t_i$ where 
\begin{equation}\label{centroidstarsig}
    \mathfrak{c}(*\sig)_i = \frac{1}{|*\sig|}\int_{*\sig} y_i \vol_{*\sig}.
\end{equation}

\begin{proposition}\label{prop:superconvergence}
    If $\mathfrak{c}(\sig)$ = $\mathfrak{c}(*\sig)$ for all $\sig\in\del_k$ such that $\sig\not\subset\partial\Omega$, then $\overline{\pol}_1\La^k(\mathcal{Z}_h^k)\subset \Kern(\Pi-J)$.
\end{proposition}
\begin{proof}
    Let $u\in\overline{\pol}_1\La^k(\mathcal{Z}_h^k)$ and $\sig\in\del_k$ such that $\sig\not\subset\partial\Omega$. By Lemma \ref{lem:kernel}, it suffices to check
    \begin{equation}\label{kernelrelation}
    \frac{1}{|\sigma|}\int_{\sigma} u = \frac{1}{|*\sigma|}\int_{*\sigma}\star u.
\end{equation}
We begin by noting that $u$ can be written as $u = u_0 + u_1$ where 
\begin{equation*}
    u_0 = \sum_{I\in\Sigma(k,n)} a_I dy_I, \qquad u_1 = \sum_{I\in\Sigma(k,n)} b_I dy_I,
\end{equation*}
and each $a_I\in\R$ while each $b_I$ is a homogeneous linear function such that $b_I(\mathfrak{c}(\sig)) = 0$.  In particular, we can then write $b_I= \sum_{i=1}^n \beta_{i}^I (y_i- \mathfrak{c}(\sig)_i)$ where each $\beta_i^I\in\R$ for any $I\in\Sigma(k,n)$.

Now, by Lemma \ref{lem:kernel}, we know that $u_0$ satisfies \eqref{kernelrelation}. It remains to consider $u_1$. Letting $S=(1,\dots,k)$, we can compute
\begin{equation*}
    \int_\sig u_1 = \int_\sig b_S\vol_\sig = \sum_{i=1}^n \beta_{i}^S\int_\sig (y_i- \mathfrak{c}(\sig)_i)\vol_\sig = \sum_{i=1}^k \beta_{i}^S (|\sig| \mathfrak{c}(\sig)_i - |\sig| \mathfrak{c}(\sig)_i) = 0,
\end{equation*}
where we used \eqref{centroidsig}. Similarly, letting $S^* = (k+1, \dots, n)$, we can compute
\begin{equation*}
    \int_{*\sig} \star u_1 = \int_{*\sig} b_{S^*} \vol_{*\sig} = \sum_{i=1}^n \beta_{i}^{S^*}\int_{*\sig} (y_i- \mathfrak{c}(*\sig)_i)\vol_{*\sig} = \sum_{i=k+1}^n \beta_{i}^{S^*} (|*\sig| \mathfrak{c}(*\sig)_i - |*\sig| \mathfrak{c}(*\sig)_i) = 0,
\end{equation*}
where we used \eqref{centroidstarsig} and our assumption that $\mathfrak{c}(\sig) = \mathfrak{c}(*\sig)$.
\end{proof}
\section{Numerical Experiments}\label{sec:numerics}
In this section, we detail numerical experiments showcasing the optimal convergence rate as well as superconvergence. We run all simulations in dimension $n=2$ for all the Hodge-Laplacians corresponding to $k\in\{0,1,2\}$. Our simulations are all run on well-centered meshes. We consider two meshes of the same domain, $\Omega$, which is the equilateral triangle with vertices $(0,0)$, $(1,0)$, and $(1/2, \sqrt{3}/2)$. One mesh consists of all equilateral triangles, while the other is a random perturbation of the vertices of the equilateral mesh done in such a way as to maintain well-centeredness. We display the two meshes in Figure \ref{fig:numerical_mesh}. 

\begin{figure}
\centering
\begin{subfigure}[b]{0.45\textwidth}
\centering
\includegraphics[width=0.8\textwidth]{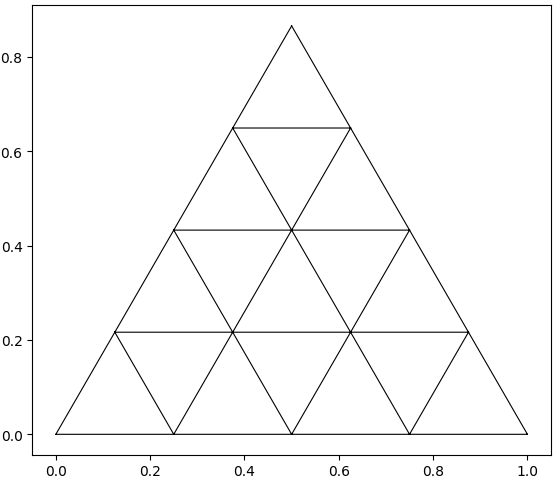}
\caption{Symmetric mesh}
\end{subfigure}
\hfill
\begin{subfigure}[b]{0.45\textwidth}
\centering
\includegraphics[width=0.8\textwidth]{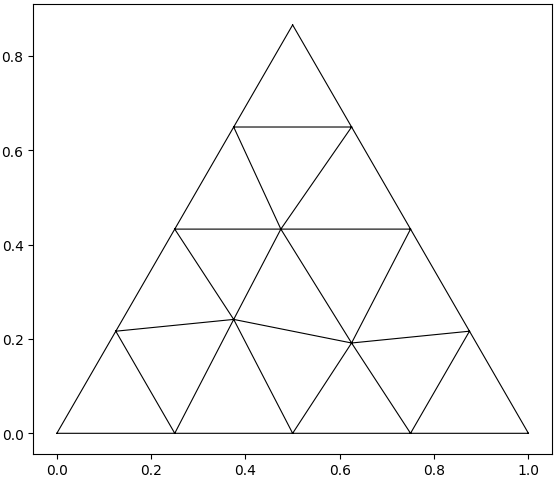}
    \caption{Perturbed mesh}
\end{subfigure}
\caption{The two meshes we use for the numerical experiments.}
\label{fig:numerical_mesh}
\end{figure}

For $k=0$ and $k=2$, we use a manufactured solution $u = 10^8 (\lambda_1 \lambda_2 \lambda_3)^5$ where $\lambda_1$, $\lambda_2$, and $\lambda_3$ are the barycentric coordinates associated to the triangle $\Omega$. For $k=1$, we choose the manufactured solution $\bu$ to be the vector with both components being $u$. Note that the manufactured solution is smooth and satisfies the conditions in Corollary \ref{corollarysuper}. Error tables are displayed in Figures \ref{fig:k0_convergence}-\ref{fig:k2_convergence}. All code is publicly available at \url{https://github.com/pratyushpotu/DEC_convergence_tests}.

Note that the symmetric mesh does indeed satisfy the assumption in Proposition \ref{prop:superconvergence} for all $k\in\{0,1,2\}$. This is easily seen for $\sig\in\del_2$ as if $\sig$ is an equilateral triangle, then $*\sig$ is the circumcenter of $\sig$, and on an equilateral triangle, the centroid is exactly the circumcenter. Moreover, on the same mesh we have $\mathfrak{c}(\sig)$ = $\mathfrak{c}(*\sig)$ for all $\sig\in\del_1$ such that $\sig\not\subset\partial\Omega$. One can see this as follows. Let $\sig\in\del_1$ be an edge. Then, if $\tau_1$ and $\tau_2$ are equilateral triangles which both contain $\sig$, the line between the centroids of $\tau_1$ and $\tau_2$ (which is exactly $*\sig$ in this case) has a midpoint at the midpoint of $\sig$. Similarly, on the same mesh we have $\mathfrak{c}(\sig)$ = $\mathfrak{c}(*\sig)$ for all $\sig\in\del_0$ such that $\sig\notin\partial\Omega$. This is because $*\sig$ is a regular hexagon centered at $\sig$ for all such interior $\sig\in\del_0$. Therefore, the lines of symmetry of $*\sig$ all pass through $\sig$ and hence $\sig$ must be the centroid of $*\sig$.

We now discuss the results presented in the error tables.

{\bf Perturbed Mesh:}
For all cases of $k\in\{0,1,2\}$, we see that the errors $\dn{e_u}$, $\dn{de_u}$, $\dn{e_\rho}$, and $\dn{de_\rho}$ (when they are defined) converge as predicted by Corollary \ref{errorestimate}. In particular, our numerical results for $\dn{de_u}$ follow our error estimate exactly. This is also seen in $\dn{e_u}$ when $k=1$ and $k=2$. When $k=0$, $\dn{e_u}$ converges at a rate of $\mathcal{O}(h^2)$. It is possible this higher order convergence could be explained via a duality argument as is standard in finite element methods, but we do not explore this here. When $k=2$, $\dn{e_\rho}$ also converges at the rate predicted by Corollary \ref{errorestimate} while we see higher order convergence when $k=1$. Finally, the error $\dn{de_\rho}$ when $k=1$ follows our estimate exactly. Note that when $k=2$, $de_\rho = 0$, and so we do not display this.

{\bf Symmetric Mesh:}
We observe superconvergence for all cases of $k\in\{0,1,2\}$ on the symmetric mesh. In particular, when $k=0$, the error $\dn{de_u}$ is superconvergent, when $k=1$ all measured errors are superconvergent, and similarly for $k=2$. Our superconvergent error estimate in Corollary \ref{corollarysuper} predicts $\mathcal{O}(h^2)$ convergence on this mesh which is observed for $\dn{e_u}$ and $\dn{de_u}$ in all cases where they are defined, while the errors $\dn{e_\rho}$ and $\dn{de_\rho}$ observe $\mathcal{O}(h^4)$ convergence. We do not have theoretical justification for this latter phenomenon.  

\begin{figure}[htbp]
  \centering
  % Subfigure: Symmetric
  \begin{subfigure}[b]{0.45\textwidth}
    \centering
    \includegraphics[width=\textwidth]{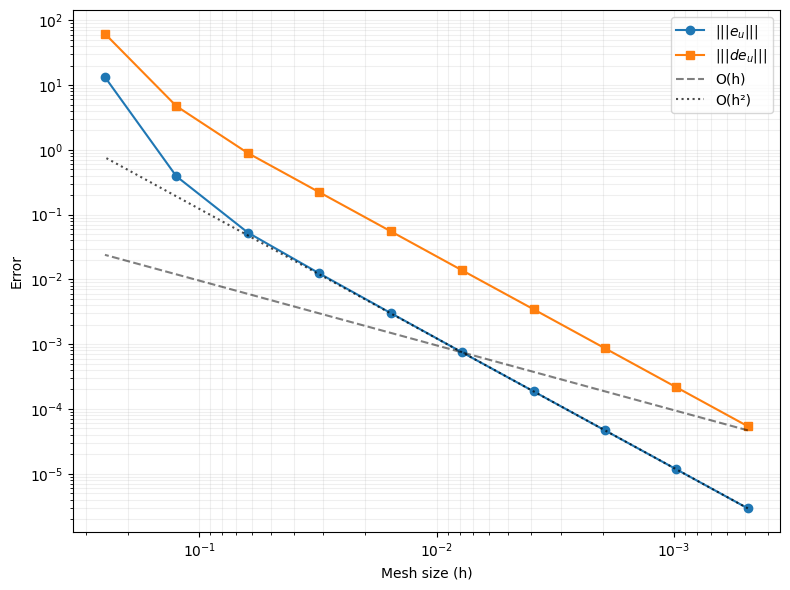}
    \caption{Symmetric mesh}
    \label{subfig:k0_symmetric}
  \end{subfigure}
  \hfill 
  \begin{subfigure}[b]{0.45\textwidth}
    \centering
    \includegraphics[width=\textwidth]{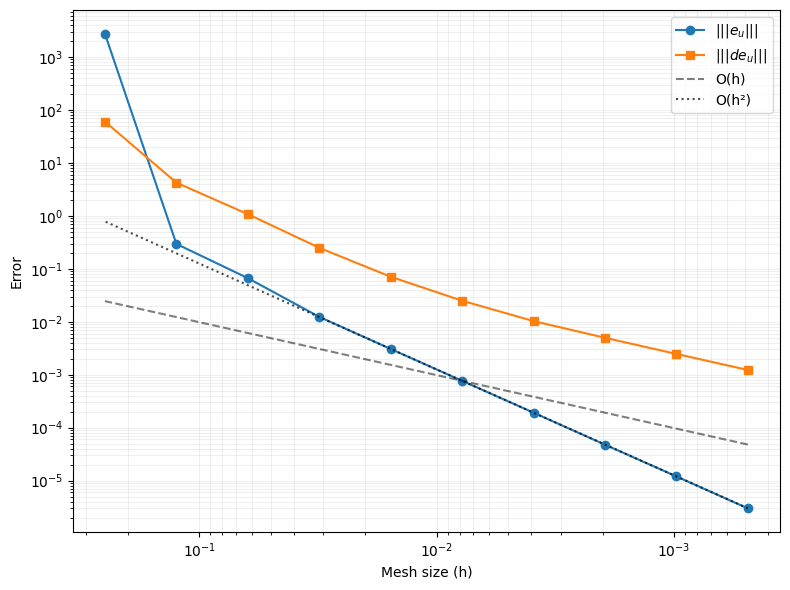}
    \caption{Perturbed mesh}
    \label{subfig:k0_perturbed}
  \end{subfigure}
  \caption{Convergence of the DEC scheme for the $k=0$ Hodge-Laplace problem}
  \label{fig:k0_convergence}
\end{figure}

\begin{figure}[htbp]
  \centering
  \begin{subfigure}[b]{0.45\textwidth}
    \centering
    \includegraphics[width=\textwidth]{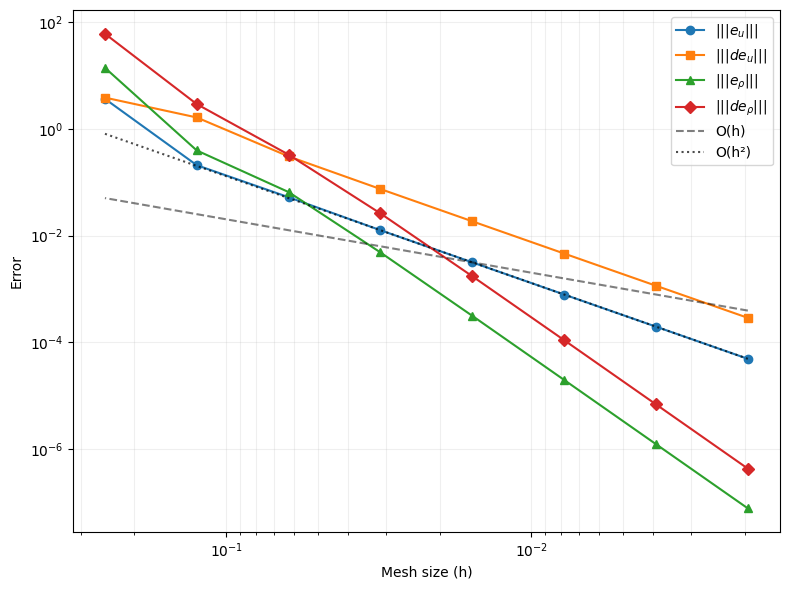}
    \caption{Symmetric mesh}
    \label{subfig:k1_symmetric}
  \end{subfigure}
  \hfill
  \begin{subfigure}[b]{0.45\textwidth}
    \centering
    \includegraphics[width=\textwidth]{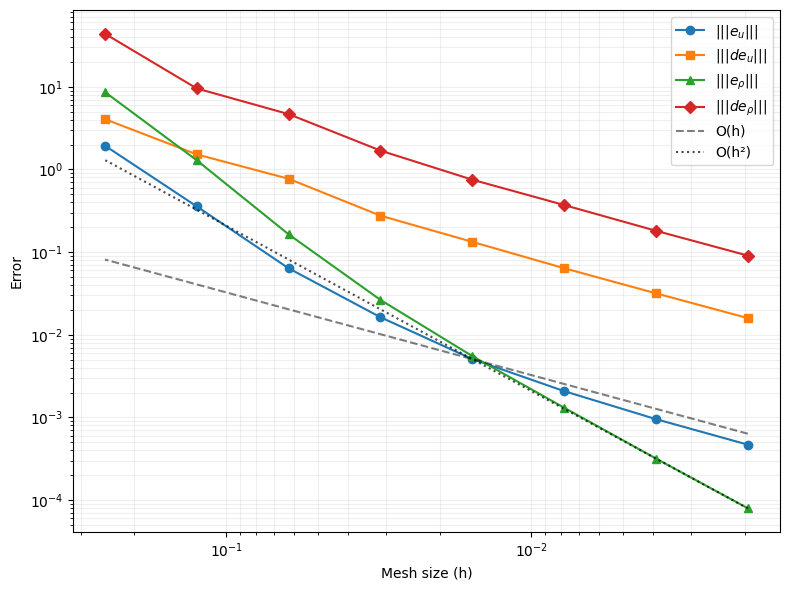}
    \caption{Perturbed mesh}
    \label{subfig:k1_perturbed}
  \end{subfigure}
  
  \caption{Convergence of the DEC scheme the $k=1$ Hodge-Laplace problem}
  \label{fig:k1_convergence}
\end{figure}

\begin{figure}[htbp]
  \centering
  \begin{subfigure}[b]{0.45\textwidth}
    \centering
    \includegraphics[width=\textwidth]{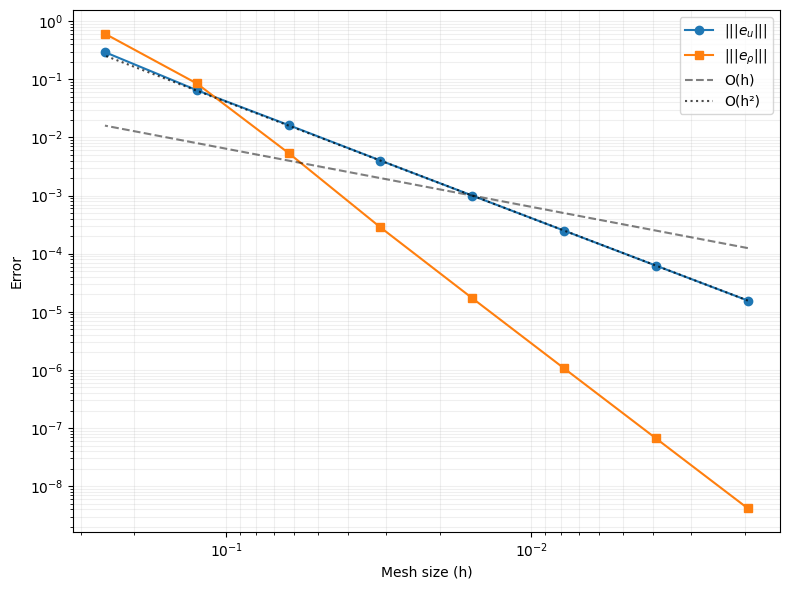}
    \caption{Symmetric mesh}
    \label{subfig:k2_symmetric}
  \end{subfigure}
  \hfill 
  \begin{subfigure}[b]{0.45\textwidth}
    \centering
    \includegraphics[width=\textwidth]{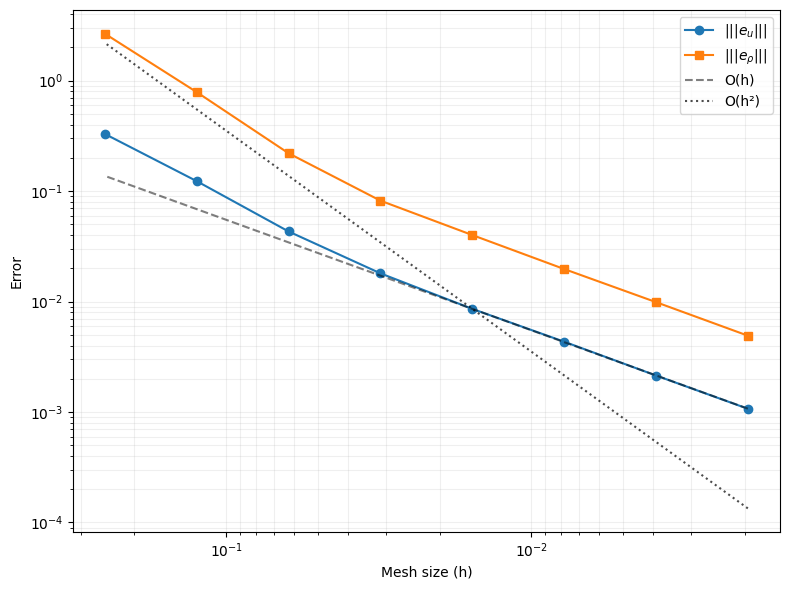}
    \caption{Perturbed mesh}
    \label{subfig:k2_perturbed}
  \end{subfigure}
  \caption{Convergence of the DEC scheme for the $k=2$ Hodge-Laplace problem}
  \label{fig:k2_convergence}
\end{figure}

\section*{Acknowledgments}
We would like to thank the anonymous referee who gave us an idea to greatly simplify the stability analysis and which allowed  us to analyze the problem with harmonic forms.
We would also like to thank Christopher Eldred for showing us the preprint \cite{Kreeft2011} where the authors define an operator similar to the $J$ operator introduced in this work in that it commutes with the discrete codifferential on finite dimensional spaces of differential forms. The setting of \cite{Kreeft2011} is based on the work \cite{BochevHyman2006} and is noticeably different from that considered in this work.

\bibliographystyle{acm_mod}
\bibliography{references}

\appendix
\section{Construction of Generalized Whitney forms}\label{GenWhitneyConstruction}
We follow the work of Christiansen (\cite{Christiansen08, Christiansen09} and \cite[Section 5]{Christiansen11}) to define finite dimensional spaces of $k$-forms satisfying the properties in Proposition \ref{GeneralizedWhitneyForms}. 

Let $T \in \Delta^{\smp}$. If $F \in \Delta_s^{\smp}(T)$ for $s \ge k$, we let $d_F^k$ be the $k$-th exterior derivative on $F$ and we let $\delta_{F,k+1}$ be its adjoint. The $k$-th Hodge-Laplace operator on $F$ is given by
\begin{equation*}
L_F^k:= d_F^{k-1} \delta_{F, k} +\delta_{F, k+1} d_F^k.
\end{equation*}

The $k$-th $\mathfrak{B}^*$ problem on $F$ is given data $\eta$ and $g$ one solves.
\begin{alignat}{2}
L_F^k w =& \eta \qquad && \text{ on } F, \label{app1}\\
\delta_{F, k} w=&0  \qquad  && \text{ on } F, \label{app2}\\
\tr_{\partial F} w= & g,   \qquad   && \text{ on } \partial F.\label{app3}
\end{alignat}
Note that \eqref{app1} becomes $\delta_{F, k+1} d_F^k w =\eta$. In particular, given $L^2$ data this problem has a unique solution for $k+1 \le s \le n$. In the case $s=k$,  $\delta_{F, k} w=0$ implies that $w$ is a multiple of the volume form $\vol_F$.

Now we define the space of the kernel of the $\mathfrak{B}^*$ problem on each face $F$ of dimension $k$ or larger of $T$. This is,
\begin{equation*}
 W_h^k(T):=\{ v \in L^2 \Lambda^k(T) : \delta_{F, k} \tr_F v=0, L_F^k \tr_F v=0, \forall F \in \Delta_s^{\smp}(T), k \le s \le n\}. \end{equation*}
Note that if $T$ is a $k$-simplex then $W_h^k(T)$  is the space of constant $k$ forms defined on $T$ since the kernel of $\delta_{F, k}$ are exactly constants forms. 

Then, we see that $v \in W_h^k(T)$, and since $\delta_{f, k} \tr_f v=0$ for all $f \in \Delta_k^{\smp}(T)$, $\tr_f v$ is a multiple of $\vol_f$ if $f \in \Delta_k^{\smp}(T)$.  Moreover, we can show that   $v \in   W_h^k(T)$ is uniquely determined by 
\begin{equation*}
    \int_f \text{tr}_f v \qquad \forall f \in \Delta_k^{\smp}(T).
\end{equation*}
Indeed, suppose that $\tr_F v$ is uniquely determined by the above dofs for every $F \in \Delta_i^{\smp}(T)$ with  $k \le i\le s-1$. Let $G \in \Delta_{s}^{\smp}$. Then $w=\tr_G v$ solves
\begin{alignat*}{2}
L_G^k w =& 0 \qquad &&\text{ on } G, \\
\delta_G^k w =&0  \qquad  &&\text{ on } G, \\
\tr_{\partial G} w= & \tr_{\partial G} v ,   \qquad   &&\text{ on } \partial G.
\end{alignat*}
Since there is a unique solution $w$, it must be uniquely determined by  $\tr_{\partial G} v$ which by our hypothesis is determined by the dofs.

We next show that the following sequence is a complex.
\begin{alignat}{4}\label{cellcomplexT}
&\mathbb{R}
\stackrel{\subset}{\xrightarrow{\hspace*{0.5cm}}}\
W_h^0(T)
&&\stackrel{d^0}{\xrightarrow{\hspace*{0.5cm}}}\
W_h^1(T)
&&\stackrel{d^1}{\xrightarrow{\hspace*{0.5cm}}}\
\cdots
&&\stackrel{d^{n-2}}{\xrightarrow{\hspace*{0.5cm}}}\
W_h^{n-1}(T)
\stackrel{d^{n-1}}{\xrightarrow{\hspace*{0.5cm}}}\
W_h^{n}(T).
\end{alignat}
 
Suppose that $v \in W_h^k(T)$ and let $w=d^k v$. Let $F \in \Delta_s^{\smp}(T)$ for $k+1 \le s \le n$. Then, since $0=L_F^k \tr_F v=\delta_{F, k+1} d_F^k \tr_F v$,  we have that $\delta_{F, k+1} \tr_F w=0$. Moreover, $L_F^{k+1} \tr_F w= \delta_{F, k+2} d_F^{k+1} \tr_F w= \delta_{F, k+2} d_F^{k+1} d_F^k \tr_F v=0$. Hence, we have that $w \in W_h^{k+1}(T)$.

One defines the global space as
\begin{equation*}
W_h^k:= \{ v \in H \Lambda^k(\Omega): v|_T \in W_h^k(T), \forall T \in \delC_k\}.    
\end{equation*}
We then see that $w \in W_h^k$ is uniquely defined by the following dofs
\begin{equation*}
    \int_f \tr_f w  \qquad \forall f \in \delC_k.
\end{equation*}
% For $f \in \Delta_k(\mathcal{C}_h)$, we define the generalized Whitney form $\phi_f \in V_h^k$, which is determined by
% \begin{equation*}
%     \int_g \text{tr}_g \phi_f= \delta_{fg}  \qquad \forall g \in \Delta_k^{\smp}(\mathcal{C}_h)
% \end{equation*}
% where $\delta_{fg}$ is the Kronecker delta. 

\end{document}